\documentclass[letterpaper,10pt,reqno]{amsart}
      \usepackage{setspace}     
     \usepackage[pdftex]{graphicx}
     \usepackage{sidecap}
\usepackage{fancyhdr, amsmath, amssymb, amsfonts, url, dsfont, mathrsfs, graphicx, epsfig,  amsthm,mathrsfs}
\usepackage{tikz}
\usepackage{wrapfig}
\usepackage{framed}
\usepackage{listings}

\usepackage{pgf}
\usepackage{tikz}
\usetikzlibrary{arrows,automata}

\usepackage{dsfont}
\usepackage{float}
\usepackage{verbatim} 
\usepackage[latin1]{inputenc}
\usepackage{cite}
\usepackage{graphicx}
\usepackage{vmargin} 
\usepackage{ifpdf} 
\usepackage{url} 
\usepackage{fancyhdr}
\usepackage{lmodern}
\usepackage{moreverb}
\usepackage{enumitem}
\usepackage{amsmath}
\usepackage{pgf}
\usepackage{tikz}
\usepackage{graphicx}
\usetikzlibrary{arrows}
\usetikzlibrary{shapes,snakes,calendar,matrix,backgrounds,folding}

\usepackage{color}
\definecolor{rltred}{rgb}{0.75,0,0}
\definecolor{rltgreen}{rgb}{0,0.5,0}
\definecolor{rltblue}{rgb}{0,0,0.75}

\usepackage{thumbpdf}
\usepackage[pdftex,colorlinks=true,urlcolor=rltred, filecolor=rltgreen, linkcolor=rltblue]{hyperref}

 \newtheorem{thm}{Theorem}[section]
\newtheorem{cor}[thm]{Corollary}
\newtheorem{lemma}[thm]{Lemma}
\newtheorem{prop}[thm]{Proposition}
\newtheorem{definition}[thm]{Definition}

\newtheorem{remark}[thm]{Remark}

\newcommand{\N}{\mathbb{N}}

\begin{document}

\title{The Wasserstein distance between stationary measures}

\begin{thanks}
{The author was partially supported by CONICYT PIA ACT172001 }
\end{thanks}

\author{Italo Cipriano} 

\address{Facultad de Matem\'aticas,
Pontificia Universidad Cat\'olica de Chile (PUC), Avenida Vicu\~na Mackenna 4860, Santiago, Chile}
\email{icipriano@gmail.com }
\urladdr{http://www.icipriano.org/}

\maketitle

\begin{abstract}

We provide explicit formulaes for the first Kantorovich-Wasserstein distance between stationary measures for iterated function schemes on the unit interval. In particular, we consider two stationary measures with different configurations of the weights associated to the same iterated function scheme with disjoint images composed of: $k$ positive contractions or $2$ contractions of different sign. We also study the case of two stationary measures associated to different iterated function schemes.

\end{abstract}

\section{Introduction}

The optimal transport problem is an engineering problem proposed by Monge in 1781. In a few words it refers to the minimization of the costs of transporting an entire collection of objects into other, where the initial and the final spacial distribution (or mass distribution, or in general any other distribution of a physical property) of the objects can be modelled by probability measures. The mathematical model considers a probability measure $\mu$ that models the objects that are being taken, a probability measure $\nu$ that models the objects that are being deposited, a transport function $T(x)=y,$ and a cost function $c(x,y).$ The transport problem corresponds to find $T$ such that $\nu=\mu\circ T^{-1}$ and such that minimizes the total cost of the transport
$$
\int c(x,T(x))d\mu(x).
$$

Applications of optimal transport include image processing, for instance comparing color distributions \cite{700, 713}, traffic control \cite{1000,1001,1002}, economics and evolution PDEs, among others. We address the reader interested in the last mentioned applications to the book \cite{1003}, and in general, for more information of optimal transport problem and further references to the books by Villani \cite{Villani1, Villani2}. We are interested in a particular version of the transport problem that satisfies the axioms of a distance function, called Wasserstein distance.

Given a Polish metric space $(\mathcal{X},d),$ and $p\in [1,\infty).$ For any two probability measures $\mu,\nu$ on $\mathcal{X},$ the Wasserstein distance of order $m$ between $\mu$ and $\nu$ is defined by
$$
W_{m}(\mu,\nu)=\inf \left\{ \left[\mathbb{E}d(X,Y)^m\right]^{\frac{1}{m}}: \mbox{law}(X)=\mu, \mbox{law}(Y)= \nu  \right\}.
$$

Wasserstein distances are important in statistics, limit theorems and approximation of probability measures \cite{150, 254, 256, 257, 282, 694, 696, 716,704}. They have been used in the study of: Statistical mechanics, specifically, in the theory of propagation of chaos \cite{81, 82, 221, 590, 624}, Boltzmann equations \cite{776, 777}, Mixing and convergence for Markov chains  \cite{231, 662, 679}, Rates of fluctuations of empirical measures \cite{695,696, 498, 8, 307, 314, 315, 479, 771, 845}, Large-time behavior of stochastic partial differential equations \cite{455, 458, 533, 605}, Hydrodynamic limits of systems of particles \cite{444}, Ricci curvature \cite{662}, Linearly rigid spaces \cite{625},Towers of measures, Bernoulli automorphisms and classification of metric spaces \cite{809}. 

Our goal is exploring the link between Wasserstein distances and fractal geometry. In order to settle our setting, we start by the definitions of iterated function system (IFS) and stationary probability measure (SPM).

An IFS is a finite set of contractions $(f_1,f_2,\ldots f_N)$ in $\mathbb{R}^d.$ Hutchinson proved in \cite{Hutchinson} that associated to an IFS there is a unique non empty and compact set that is invariant under the IFS, that is, there is a unique $\mathcal{S}\subset \mathbb{R}^d$ non empty and compact such that
$$
\mathcal{S}=\cup_{i=1}^N f_i(S).
$$
This invariant set is called the attractor of the IFS. Examples are the Cantor ternary set on the real line, and the Sierpinski gasket in $\mathbb{R}^2.$ Associated to each IFS $f=(f_1,f_2,\ldots f_N)$ in $\mathbb{R}^d$ and a probability vector $p=(p_1,p_2,\ldots,p_N)\in [0,1]^N$ there is a unique probability measure $\mu=\mu^{(f,p)}$ such that
$$
\mu(A)=\sum_{i=1}^N p_i \mu (f_i^{-1}(A)) \mbox{ for all } A\in \mathcal{B},
$$
where $\mathcal{B}$ are the subset of Borel of $\mathbb{R}^d.$ This probability measure is called a SPM and its existence and unicity is proved in \cite{Hutchinson}.

Fraser initiated in \cite{Fraser} the study of the Wasserstein distances between stationary measures. An explicit formula for $W_{1}\left(\mu^{(f,p)},\mu^{(f,q)}\right)$ and an upper bound for $W_{2}\left(\mu^{(f,p)},\mu^{(f,q)}\right)$ were obtained in \cite{Fraser} when $f=(f_1,f_2), p,q \in (0,1)^2$ and $f_1(x)=\rho x+t_1,$ $f_2(x)=\rho x+t_2$ where $\rho\in (0,1/2],$ $t_1\in [0,1-2\rho]$ and $t_2 \in [t_1+\rho,1-\rho].$ An explicit formula of $W_{1}\left(\mu^{(f,p)},\mu^{(f,q)}\right)$ in the case of  $f_1(x)=\rho_1x+t_1,$ $f_2(x)=\rho_2x+t_2,$ with $\rho_1,\rho_2>0$ and $f_1(0,1)\cap f_2(0,1)=\emptyset$ was obtained in \cite{Mark_Pollicott_Italo_Cipriano}, together with a good approximation of $W_{1}\left(\mu^{(f,p)},\mu^{(f,q)}\right)$ when $f_1, f_2$ are positive Lipschitz contractions of the open interval such that $f_1(0,1)\cap f_2(0,1)=\emptyset.$\\

The main results of these notes can be summarised as follows.
\begin{enumerate}
\item We obtain a good approximation of $W_{1}\left(\mu^{(f,p)},\mu^{(f,q)}\right)$ where $f=(f_1,\ldots,f_k)$ are $k$ positive Lipschitz contractions of the unit interval such that $f_i(0,1)\cap f_j(0,1)=\emptyset$ for every $i\neq j.$ We obtain an explicit formula in the particular case $f_ix=\rho_ix+t_i$ under certain conditions on the weights. This solves a problem proposed in \cite{Fraser}.
\item We obtain a good approximation of $W_{1}\left(\mu^{(f,p)},\mu^{(g,q)}\right)$ where $f=(f_1,f_2)$ and $g=(g_1,g_2)$ are each  positive Lipschitz contractions of the unit interval such that $f_1[0,1]\cap f_2[0,1]= g_1[0,1]\cap g_2[0,1]= \emptyset.$ We obtain an explicit formula in the particular case $f_ix=\alpha_ix+a_i$ and $g_ix=\beta_ix+b_i$ for $i=1,2,$ under certain conditions on the weights.
\item We obtain a good approximation of $W_{1}\left(\mu^{(f,p)},\mu^{(f,q)}\right)$ where $f=(f_1,f_2),$ $f_1$ is a positive Lipschitz contraction and $f_2$ is a negative Lipschitz contraction of the unit interval such that $f_1[0,1]\cap f_2[0,1]=\emptyset,$ under certain condition of symmetry on the weights.
\end{enumerate}

The main difficulty in the proofs is an accurate description of the intersections of certain non-classics Cantor staircases. The author believes the ideas of the proof are general for IFS with more contractions and certain symmetries on the weights functions. The author also believes the problem of estimating $W_m\left(\mu^{(f,p)},\mu^{(g,q)}\right)$ in full generality is worth investigating, although it remains a relatively new subject of study.\\

The paper is written in three main chapters: the first include the results, the second the proofs, and the third a few computational examples for each theorem.

\section{Results}

First, we define positive and negative Lipschitz contractions, and recall some useful results for the case of $\mathbb{R}^d$ with $d=1.$ Second, we state the main results.

\subsection{Definitions and useful results}

\begin{definition}[Lipschitz contractions]
We say that $f_i: \mathbb R \to \mathbb R$ with $i=1,\ldots, k$ is a family of Lipschitz contractions if
$$Lip(f_i):=\sup_{x,y\in \mathbb{R}}\frac{|f_i(x)-f_i(y) |}{|x-y|}<1.$$
We can assume without loss of generality
that $f_i: [0,1] \to [0,1]$ and 
\begin{equation}\label{eq_cero}
\max\{f_i(0),f_i(1)\} \leq \max\{f_{i+1}(0),f_{i+1}(1)\}
\end{equation}
with a suitable choice of coordinates. We say that $f=(f_1,\ldots,f_k)$ is an IFS of Lipschitz contractions. We say that a Lipschitz contraction $f_i: \mathbb R \to \mathbb R$ is positive if $f_i$ is differentiable in $\mathbb{R}$ and $\frac{f_i}{dx}>0,$ and we say it is negative if $\frac{f_i}{dx}<0.$ We say that $f=(f_1,\ldots,f_k)$ is an IFS of positive Lipschitz contractions if each $f_i$ is a positive contraction.
\end{definition}

We observe an alternative re-writing of the equation for the SPM.

\begin{remark}
Given an IFS $f=(f_1,\ldots,f_k)$ of Lipschitz contractions and $p=(p_1,\ldots,p_k)\in (0,1)^k,$ then $\mu=\mu^{(f,p)}$ is the SPM iff
\begin{equation}\label{eq_uno}
\int_0^1 \phi(x)d \mu(x)=  \sum_{i=1}^k p_i \int_0^1 \phi \circ f_i (x) d\mu(x)
\end{equation}
for every continuous function $\phi:[0,1]\to\mathbb R.$
\end{remark}

We recall the Kantorovich-Rubinstein duality theorem that gives the following reformulation of $W_1(\mu,\nu).$ 
\begin{thm}
If $\mu$ and $\nu$ are probability measures on $\mathbb{R}$ with compact support
$$
W_1(\mu,\nu)=\sup\left\{\int_{-\infty}^{\infty} \phi(x) d \mu(x)- \int_{-\infty}^{\infty} \phi(x) d\nu(x): Lip(\phi)\leq 1 \right\}.
$$
\end{thm}

We have an useful characterisation of this distance.

\begin{thm}[Dall'Aglio-Vallender]\label{Bochi}
Let $\mu$ and $\nu$ be probability measures on $\mathbb R.$ Then 
\begin{equation}\label{eq_dos}
W_1(\mu,\nu)=\int_{-\infty}^{\infty} |F(t)-G(t)|dt,
\end{equation}
where $F$ and $G$ are the cumulative distribution functions of $\mu$ and $\nu.$
\end{thm}

A proof can be found in \cite{Aglio} and \cite{Vallander}. In this paper, we only require a version for $\mu$ and $\nu$ probability measures on $[0,1].$ An elementary proof for this case was proposed by Jairo Bochi and it is included in \cite{Mark_Pollicott_Italo_Cipriano}.

\subsection{Main results}
  
\begin{thm}\label{kmaps}
Let $f=(f_1,\ldots f_k)$ be an IFS of positive Lipschitz contractions on the unit interval such that
\begin{equation}\label{eq_tres}
f_i(0,1)\cap f_j(0,1)=\emptyset \mbox{ for all }i,j=1,\ldots,k, i\neq j.
\end{equation}
If $(p,q)$ is a pair of probability vectors in $(0,1)^k$ such that 
\begin{equation}\label{eq_cuatro}
\sum_{i=1}^m p_i -  q_i \in \mathcal{A} \mbox{ for every }m=1,2,\ldots,k,
\end{equation}
where $\mathcal{A}=\mathbb R^{\geq 0}$ or $\mathcal{A}=\mathbb R^{\leq 0},$ then
\begin{equation}\label{eq_cinco}
W_1\left(\mu^{(f,p)},\mu^{(f,q)}\right)=\left| \int x d \left(\mu^{(f,p)}-\mu^{(f,q)}\right)(x)\right|.
\end{equation}
\end{thm}

\begin{remark}
If $k=2,$ then the condition (\ref{eq_cuatro}) is always satisfied.
\end{remark}

An immediate consequence of this theorem is an explicit formulae for the Wasserstein distance between stationary measures of affine maps.

\begin{cor}\label{kmapscor}
Let $f_i:[0,1]\to [0,1]$ be defined by $f_i x=\rho_i x+t_i,$ where $\rho_i\in (0,1)$ and $t_i\in [0,1)$ for $i=1,\ldots,k.$ Suppose that $\{f_1,\ldots,f_k\}$ satisfies (\ref{eq_cero}) and (\ref{eq_tres}). If $(p,q)$ is a pair of probability vectors in $(0,1)^k$ that satisfies (\ref{eq_cuatro}), then 
\begin{equation}\label{eq_seis}
W_1\left(\mu^{(f,p)},\mu^{(f,q)}\right)=\left| \frac{\sum_{i}p_i t_i}{1-\sum_{i} p_i\rho_i}-\frac{ \sum_{i} q_i t_i}{1-\sum_{i} q_i\rho_i} \right|.
\end{equation}
\end{cor}

The next theorem gives a good estimation of the first Wasserstein distance between SPMs associated to possibly different IFS with possibly different configurations of the weights.

\begin{thm}\label{teo_dos}
Let $f=(f_1,f_2)$ and $g=(g_1,g_2)$ be IFS of positive Lipschitz contractions on the unit interval. Suppose that $f=(f_1,f_2)$ satisfies (\ref{eq_cero}) and (\ref{eq_tres}), and that $f_1(0)=g_1(0),$ $f_2(0)=g_2(0),$ $g_1(x)\leq f_1(x),$ $g_2(x)\leq f_2(x)$ for all $x\in[0,1].$ If $(p,q)$ is a pair of probability vectors $p=(p_1,1-p_1)$ and $q=(q_1,1-q_1)$ such that $p_1\leq q_1,$ then 
$$W_1\left(\mu^{(f,p)},\mu^{(g,q)}\right)=\int_0^1 x d\left(\mu^{(g,q)}-\mu^{(f,p)}\right)(x).$$
 \end{thm}

Again, we can write an explicit formulae for the first Wasserstein distance between SPMs of affine maps with positive slope.

\begin{cor}\label{cor_may3}
Let $f_i,g_i:[0,1]\to [0,1]$ be defined by $f_i x=\alpha_i x+t_i,$ $g_i x= \beta_i x+t_i$ where $\rho_i\in (0,1),$ $\beta_i\in(0,\rho_i]$ and $t_i\in [0,1)$ for $i=1,2.$ Suppose that $f=(f_1,f_2)$ satisfies (\ref{eq_cero}) and (\ref{eq_tres}). If $(p,q)$ is a pair of probability vectors $p=(p_1,1-p_1)$ and $q=(q_1,1-q_1)$ such that $p_1\leq q_1,$ then 
$$
W_1\left(\mu^{(f,p)},\mu^{(g,q)}\right)=\frac{\sum_{i}q_i t_i}{1-\sum_{i} q_i\beta_i}-\frac{\sum_{i}p_i t_i}{1-\sum_{i} p_i\rho_i}.
$$
\end{cor}



A natural question is what can be said in the case of non necessarily positive Lipschitz contractions. The following theorem consider the case of the IFS  $f^r:=(f^r_1,f^r_2)$ for $r\in (2,\infty)$ defined by 
$$
\begin{aligned}
f^r_1 x =& \frac{x}{r}\\
f^r_2 x =& 1-\frac{x}{r}.
\end{aligned}
$$
 
\begin{thm}\label{teo_cuatro}
Let $k\in\mathbb{N}$ and $r\in (2k+1,\infty).$ Then for $p=(p_1,p_2)=(\frac{1}{2k+1},\frac{2k}{2k+1})$ and $q=(p_2,p_1)$ we have that 
$$W_1\left(\mu^{(f^r,p)},\mu^{(f^r,q)}\right)=\int_0^1 c_{r}(x) d\left(\mu^{(f^r,p)}-\mu^{(f^r,q)}\right) (x),$$
where
$$
c_{r}(x):=
\begin{cases}
-x & \mbox{ if }x<\frac{r^2}{r^2+1}\\
x & \mbox{ if }x>\frac{r^2}{r^2+1}.
\end{cases}
$$
\end{thm}

The following proposition considers the case of a particular IFS $g^r$ with only positive Lipschitz contractions and the IFS $f^r.$ Let $g^r:=(g^r_1,g^r_2)$ for $r\in (2,\infty)$ defined by 
$$
\begin{aligned}
g^r_1 x =& \frac{x}{r}\\
g^r_2 x =& \frac{x}{r}+\frac{r-1}{r}.
\end{aligned}
$$

\begin{prop}\label{prop_cinco}
Let $p=(p_1,p_2)$ be a probability vectors in $(0,1)^2$ and $r\in(2,\infty),$ then
$$
W_1\left(\mu^{(f^r,p)},\mu^{(g^r,p)}\right)=
\begin{cases}
\int_0^1 x d\left(\mu^{(g^r,p)}-\mu^{(f^r,p)}\right)(x)  & \mbox{ if } p_1\in (0,\frac{1}{2}) \\
0  & \mbox{ if } p_1=\frac{1}{2} \\
\int_0^1 x d\left(\mu^{(f^r,p)}-\mu^{(g^r,p)}\right)(x)  & \mbox{ if } p_1\in (\frac{1}{2},1). 
\end{cases}
$$
\end{prop}

\begin{remark}
The author believes Theorem $\ref{teo_cuatro}$ is still valid under the weaker hypotheses $r\in (2,\infty),$ $p_1\in (0,1)$ and $p=(p_1,p_2), q=(p_2,p_1).$
\end{remark}

\section{Proofs}

\subsection{Proof of Theorem \ref{kmaps}}

We start by proving a lemma similar to Lemma 3.2. in \cite{Mark_Pollicott_Italo_Cipriano}. For this, let us introduce a definition.

\begin{definition}
Let $f=(f_1,\ldots f_k)$ be an IFS of positive Lipschitz contractions on $[0,1]$ and $(p,q)$ be a pair of probability vectors in $(0,1)^k.$ We define the function $D^{(p,q)}:[0,1]\to[0,1]$ by $$D^{(p,q)}(x):=(\mu^{(f,p)}-\mu^{(f,q)})[0,x].$$
\end{definition}

Our lemma is the following.

\begin{lemma}\label{lemma_2_plus}
Let $f=(f_1,\ldots f_k)$ be an IFS of positive Lipschitz contractions on the unit interval $[0,1]$ that satisfy (\ref{eq_cero}) and $f_i[0,1]\cap f_j[0,1]=\emptyset \mbox{ for all }i,j=1,\ldots,k, i\neq j$. If $(p,q)$ is a pair of probability vectors in $(0,1)^k$ that satisfies (\ref{eq_cuatro}), then $D^{(p,q)}$ does not change of sign.
\end{lemma}

\begin{proof}
Suppose without lost of generality that 
$$
\sum_{i=1}^m p_i - \sum_{i=1}^m q_i\geq 0 \mbox{ for every }m=1,2,\ldots,k.
$$
We can find a probability vector $r:=(r_1,\ldots,r_{2k-1})\in (0,1)^{2k-1}$ such that
$$
\begin{aligned}
p_1&=\sum_{i=1}^{l_1}  r_i,        &q_1&= \sum_{i=1}^{t_1}r_i,\\
p_2&=\sum_{i=l_1+1}^{l_2}  r_i,&q_2&=\sum_{i=t_1+1}^{t_2}r_i,\\
\vdots\\
p_k&=\sum_{i=l_{k-1}+1}^{l_k}  r_i,&q_k&=\sum_{i=t_{k-1}+1}^{t_k}r_i,
\end{aligned}
$$
where $1\leq l_1<l_2<\cdots<l_k=2k-1,$ $1\leq t_1<t_2<\cdots<t_k=2k-1,$ and $t_i\leq l_i$ for $i=1,\ldots,k.$ We now consider the shift spaces $\Sigma_{k}:=\{1,\ldots,k\}^{\mathbb{Z}_+}$ and $\Sigma_{2k-1}:=\{1,\ldots,2k-1\}^{\mathbb{Z}_+}.$ On $\Sigma_{k}$ we define the Bernoulli measure $\mu_{p}:=\prod_0^{\infty}(p_1,\ldots,p_k)$  and $\mu_{q}:=\prod_0^{\infty}(q_1,\ldots,q_k).$ On $\Sigma_{2k-1}$ we define the Bernoulli measure $\mu_{r}:=\prod_0^{\infty}(r_1,\ldots,r_{2k-1}).$ We proceed now to define two projections $\pi_{p},\pi_{q}:\Sigma_{2k-1}\to\Sigma_k$ by 

$$
\pi_{p}\left((i_n)_{n=0}^\infty \right)=  (j_n)_{n=0}^\infty \hbox{ where } j_n = 
\begin{cases}
1 & \hbox{ if } i_n \in [1,l_1], \\
2 & \hbox{ if } i_n \in [l_1+1,l_2], \\
\vdots\\
k & \hbox{ if } i_n \in [l_{k-1}+1,l_{k}],
\end{cases}
$$
and
$$
\pi_{q}\left((i_n)_{n=0}^\infty \right)=  (j_n)_{n=0}^\infty \hbox{ where } j_n = 
\begin{cases}
1 & \hbox{ if } i_n \in [1,t_1], \\
2 & \hbox{ if } i_n \in [t_1+1,t_2], \\
\vdots\\
k & \hbox{ if } i_n \in [t_{k-1}+1,t_{k}].
\end{cases}
$$

There is a natural bijection between $\Sigma_k$ and the limit set $\Lambda$ of $f_1,\ldots,f_k$ (recall by \cite{Hutchinson}, the limit set $\Lambda$ is the unique non-empty set that is invariant for $f_1,\ldots,f_k,$ i.e. $\cup_{i=1}^kf_i\Lambda=\Lambda$). This bijection is given by the map $f:\Sigma_k\to\Lambda,$ defined by $f(x_0,x_1,\dots)=\lim_{n\to\infty}f_{x_0}\circ\cdots\circ f_{x_n}([0,1]).$

The assumption that $f_1,\ldots,f_k$ are increasing functions imply that if $x,y\in \Sigma_k$ satisfy $x\in \pi_p \circ\pi^{-1}_q y,$ then $x\leq y$ with respect to the lexicographic order. Therefore, the map $g:\Lambda\to\Lambda,$ defined by 
$$
g(x):=\sup\{y\in \Lambda: x=f\circ \pi_p \circ\pi^{-1}_q\circ f^{-1}(y)\},
$$
is monotone, moreover, $g(x)\leq x$ for every $x\in \Lambda.$ We extend the map $g$ to the unit interval by $\tilde{g}:[0,1]\to[0,1],$ $\tilde{g}(t):=\sup\{g(x):x\in \Lambda, x\leq t\}.$ Clearly, we have $\tilde{g}(x)\leq x$ for every $x\in [0,1].$ Finally, we observe that 
$ \mu^{(f,p)}=\mu_{p}\circ f^{-1},$ $ \mu^{(f,q)}=\mu_{q}\circ f^{-1},$ $ \mu_{p}=\mu_{r}\circ \pi_{p}^{-1}$ and $\mu_{q}=\mu_{r}\circ \pi_{q}^{-1}.$ Finally, if $t\in [0,1],$ $$\mu^{(f,q)}[0,t]\leq \mu^{(f,p)}[0,\tilde{g}(t)]\leq \mu^{(f,p)}[0,t],$$ which concludes the proof.
\end{proof}

The proof of Theorem \ref{kmaps} under the assumption $f_i[0,1]\cap f_j[0,1]=\emptyset \mbox{ for all }i,j=1,\ldots,k, i\neq j$ is then direct from Theorem \ref{Bochi} and Lemma \ref{lemma_2_plus}. The proof under the weaker assumption (\ref{eq_tres}) is straightforward from the proof of Lemma 3.3 part 2. in \cite{Mark_Pollicott_Italo_Cipriano}. The proof of Corollary \ref{kmapscor} follows from Theorem \ref{kmaps} and the following lemma.

\begin{lemma}
Let $f_i:[0,1]\to [0,1]$ be defined by $f_i x=\rho_i x+t_i,$ where $\rho_i\in (-1,1)$ and $t_i\in [0,1)$ for $i=1,\ldots,k.$ Suppose that $\{f_1,\ldots,f_k\}$ satisfies (\ref{eq_cero}) and $p$ is a probability vectors in $(0,1)^k.$ Then
\begin{equation}\label{eq1_24_march_2016}
\int_0^1 x d\mu^{(f,p)}(x)=\frac{\sum_i p_i t_i}{1-\sum_i p_i\rho_i}.
\end{equation}
\end{lemma}

\begin{proof}
In order to prove (\ref{eq1_24_march_2016}) we use the definition of stationary measure to obtain
$$
\begin{aligned}
\int_0^1 x d\mu^{(f,p)}(x)&=\sum_i p_i\int_0^1f_i (x) d\mu^{(f,p)}(x)\\
&=(\sum_i p_i\rho_i)\int_0^1 x d\mu^{(f,p)}(x)+ \sum_i p_i t_i,\\
\end{aligned}
$$
then
$$
 \int_0^1 x d\mu^{(f,p)}(x)=\frac{\sum_i p_i t_i}{1-\sum_i p_i\rho_i},
$$
as claimed.
\end{proof}

\subsection{Proof of Theorem \ref{teo_dos}}

The proof of Theorem \ref{teo_dos} follows from the following lemma.

\begin{lemma}\label{lem_clave}
Let $f_i,g_i:[0,1]\to [0,1]$ be defined by $f_i x=\alpha_i x+t_i,$ $g_1 x= \beta x+t_1$ and $g_2= f_1,$ where $\rho_i\in (0,1),$ $\beta\in(0,\rho_1)$ and $t_i\in [0,1)$ for $i=1,2.$ Suppose that $f=(f_1,f_2)$ satisfies (\ref{eq_cero}) and (\ref{eq_tres}). If $(p,q)$ is a pair of probability vectors $p=(p_1,1-p_1)$ and $q=(q_1,1-q_1)$ such that $p_1<q_1,$ then 
$$
W_1\left(\mu^{(f,p)},\mu^{(g,q)}\right)=\left| \frac{p_1 t_1+p_2 t_2}{1- p_1\rho_1-p_2\rho_2}-\frac{q_1 t_1+q_2 t_2}{1- q_1 \beta-q_2\rho_2} \right|.
$$
\end{lemma}

\begin{proof}
For notational convenience, assume without lost of generality that $f_1(0)=0$ and $f_2(1)=1,$ consequently, $g_1(0)=0$ and $g_2(1)=1.$ Let call $\Lambda^f$ the limit set of $f=(f_1,f_2)$ and $\Lambda^g$ the limit set of $g=(g_1,g_2).$ Consider the bijections $\psi_f:\Sigma_2\to \Lambda^f$ given by $\psi_f(x_0,x_1,\ldots)=\lim_{n\to\infty}f_{x_0}\circ f_{x_1}\circ\cdots\circ f_{x_n}([0,1]) $ and $\psi_g:\Sigma_2\to \Lambda^g$ given by $\psi_g(x_0,x_1,\ldots)=\lim_{n\to\infty}g_{x_0}\circ g_{x_1}\circ\cdots\circ g_{x_n}([0,1]).$ We define the bijection $\psi: \Lambda^f\to \Lambda^g$ given by $\psi(x)=\psi_g\psi_f^{-1}(x).$ We will show that $\psi(x)\leq x$ for every $x\in \Lambda^f.$ For this, we consider a decomposition of $\Lambda^f$ and $\Lambda^g$ in non-disjoints sets of ordered points, that we call layers. The first layer of $\Lambda^f$ is $L^f_1=\{0,1\},$ the second is defined by $L^f_2=\{f_1(0),f_1(1),f_2(0),f_2(1)\},$ the third by  $L^f_3=\{f_1\circ f_1(0),f_1 \circ f_1(1), f_1 \circ f_2(0),f_1 \circ f_2(1), f_2\circ f_1(0),f_2 \circ f_1(1), f_2 \circ f_2(0),f_2 \circ f_2(1) \},$ etc...We observe that the if the $n$-th layer is defined by $L^f_n=\{x_0^{n},x_1^{n},\ldots, x_{2^n-1}^n\},$ with $x_i^{n}<x_{i+1}^{n}$ for every $i=0,1,\ldots, 2^{n}-2,$ then $L^f_{n+1}=\{x_0^{n+1},x_1^{n+1},\ldots, x_{2^{n+1}-1}^{n+1}\}$ where $f_1 (x_i^{n})=x_i^{n+1},$ $f_2 (x_i^{n})=x_{i+2^n}^{n+1},$ and in particular $x_i^{n+1}<x_{i+1}^{n+1}$ for every $i=0,1,\ldots, 2^{n+1}-2.$ Analogously, we define the $n+1$-th layer of $\Lambda^g$ inductively by $L^g_{1}=\{0,1\},$ $L^g_{n}=\{y_0^{n},y_1^{n},\ldots, y_{2^n-1}^n\}$ with $y_i^{n}<y_{i+1}^{n}$ for every $i=0,1,\ldots, 2^{n}-2,$ and $L^g_{n+1}=\{y_0^{n+1},y_1^{n+1},\ldots, y_{2^{n+1}-1}^{n+1}\}$ where $g_1 (y_i^{n})=y_i^{n+1},$ $g_2 (x_i^{n})=y_{i+2^n}^{n+1}.$ We observe that for every $n\in\mathbb{N}$ and $i\in \{0,2,\ldots,2^{n-1}\}$ the biyection $\psi: \Lambda^f\to \Lambda^g$ satisfies that $\psi(x_i^n)=y_i^n.$ In order to prove that $\psi(x)\leq x$ for every $x\in \Lambda^f$ we will use induction in the number of the layers. For the base case $n=1,$ we have that $\psi(0)=0\leq 0$ and $\psi(1)=1\leq 1.$ Assume that $\psi(x)\leq x$ for every $x\in \L_n^f,$ we will prove that $\psi(x)\leq x$ for every $x\in \L_{n+1}^f.$ Let $x\in \L_{n+1}^f,$ therefore $x=x_i^{n+1}$ for some $i\in\{0,1,\ldots,2^{n+1}-1\}$ and $\psi(x)= \psi(x_i^{n+1})=y_i^{n+1}\in \L_{n+1}^g.$ There are two options for the index $i:$
\begin{enumerate}
\item If $i\in\{0,1,\ldots, 2^{n}-1\},$ then $x_i^{n+1}=f_1 (x_i^{n}),$ $y_i^{n+1}=g_1 (y_i^{n}).$ By hypothesis $\psi(x_i^{n})=y_i^{n}\leq x_i^{n},$ therefore, by monotonicity of $g_1$ we have that $y_i^{n+1}=g_1(y_i^{n})\leq g_1(x_i^{n}),$ and by the fact that $g_1(x)\leq f_1(x)$ in $[0,1],$ we conclude that $\psi(x_i^{n+1})=y_i^{n+1}\leq f_1(x_i^{n})= x_i^{n+1}.$
\item If $i\in\{2^{n},2^{n}+1,\ldots, 2^{n+1}-1\},$ then $x_i^{n+1}=f_2 (x_{i-2^n}^{n}),$ $y_i^{n+1}=g_2 (y_{i-2^n}^{n}).$ By hypothesis $\psi(x_i^{n})=y_i^{n}\leq x_i^{n},$ therefore, by monotonicity of $g_1$ we have that $y_i^{n+1}=g_2(y_{i-2^n}^{n})\leq g_2(x_{i-2^n}^{n}),$ and by the fact that $g_2(x)= f_2(x)$ in $[0,1],$ we conclude that $\psi(x_i^{n+1})=y_i^{n+1}\leq f_2(x_{i-2^n}^{n})= x_i^{n+1}.$
\end{enumerate}
Let $x\in[0,1],$ we have that $\mu^{f,p}[0,x]=\mu^{f,p}[0,x^*]$ for some $x^*\in\Lambda^f$ with $x^*\leq x.$ By definition of the function $\psi,$ we have that $\mu^{f,p}[0,x^*]=\mu^{g,p}[0,\psi(x^*)].$ By the the fact that $\psi$ satisfies $\psi(x)\leq x,$ we have that $\mu^{g,p}[0,\psi(x^*)]\leq \mu^{g,p}[0,x^*].$ Because $x^*\leq x,$ we have that $\mu^{g,p}[0,x^*]\leq \mu^{g,p}[0,x].$ Finally, because $p<q,$ we have that $\mu^{g,p}[0,x]<\mu^{g,q}[0,x],$ and therefore, $\mu^{f,p}[0,x]\leq \mu^{g,q}[0,x].$
We now have that the map $D(x)=(\mu^{g,q}-\mu^{f,p})[0,x]$ does not change sign, therefore, using Theorem \ref{Bochi}, we conclude that $W_1(\mu^{f,p},\mu^{g,q})=\int_0^1 x d(\mu^{g,q}-\mu^{f,p})(x).$ We use now that $\int_0^1 x d\mu^{(f,p)}(x)=\frac{ p_2 t_2}{1-\sum_i^2 p_i\alpha_i}$ and $\int_0^1 x d\mu^{(g,q)}(x)=\frac{ q_2 t_2}{1-q_1 \beta-q_2 \alpha_2}$ to obtain that 
$$
W_1(\mu^{f,p},\mu^{g,q})=\frac{ q_2 t_2}{1-q_1 \beta-q_2 \alpha_2}-\frac{ p_2 t_2}{1-\sum_i ^2 p_i\alpha_i},
$$
which finished the proof.
\end{proof}

A direct consequence is Corollary \ref{cor_may3}. We observe that we did not not use the fact that the maps $f_i,g_i$ are affine until the last few lines of the proof, indeed, the condition of positive Lipschitz contractions suffices in order to prove that $W_1(\mu^{f,p},\mu^{g,q})=\int_0^1 x d(\mu^{g,q}-\mu^{f,p})(x).$  Therefore, the same proof can be used to prove Theorem \ref{teo_dos}.




\subsection{Proof of Theorem \ref{teo_cuatro}}

Along this subsection we consider the IFS $f^r=(f^r_1,f^r_2)$ defined by $f^r_1 x=\frac{x}{r}$ and $f^r_2 x=1-\frac{x}{r},$ where $r\in (2,\infty).$ The cumulative distribution function of the stationary measures $\mu^{(f^r,p)}$ will be denoted by $F_{r,p}.$ Observe that the function $F_{r,p}$ is a Cantor map.\\

The first lemma is a topological result for the IFS $f^r.$  Let $g^r=(g^r_1,g^r_2)$ defined by $g^r_1 x=\frac{x}{r}$ and $g^r_2 x=\frac{x}{r}+\frac{r-1}{r},$ where $r\in (2,\infty).$ Let $\Sigma_{2}^k:=\{1,2\}^k$ for $k\in\mathbb{N}$ and $\Sigma_{2}^*:=\cup_{k\in\mathbb{N}}\{1,2\}^{k}.$ Given $w=(w_1,\ldots,w_k)\in\Sigma_{2}^k,$ define $|w|=k,$ $f^r_{w}:=f^r_{w_1}\circ f^r_{w_2}\circ \cdots\circ f^r_{w_k}$ and $g^r_{w}:=g^r_{w_1}\circ g^r_{w_2}\circ \cdots\circ g^r_{w_k}.$ Let denote by $<_{lex}$ the lexicographic order in $\Sigma_{2}^*:$  $a <_{lex} b$  if either $a$ is a prefix of $b$ or there exists words $u, v, w$ (possibly empty) such that $a = u1v,$ $b = u2w.$ We define the different order $\prec$ in $\Sigma_{2}^*$ by: $a \prec b$ if either $a$ is a prefix of $b$ or there exists words $u, v, w$ (possibly empty) such that $a = u1v,$ $b = u2w,$ and the number of $2$ that appears in $u$ is even (or zero) or $a = u2v,$ $b = u1w,$ and the number of $2$ that appears in $u$ is odd. For example, $1\prec 2,$ $11\prec 12\prec 22\prec 21,$ $111\prec 112\prec 122\prec 121\prec 221\prec 222\prec 212\prec 211,$ $1111\prec 1112\prec  1122\prec 1121\prec  1221\prec 1222\prec  1212\prec 1211\prec  2211\prec 2212\prec  2222\prec 2221\prec  2121\prec 2122\prec  2112\prec 2111.$\\

There is a characterisation of the order $\prec$ given by the following lemma. 

\begin{lemma}\label{lem_de_los_K}
Let $K_n$ for $n\in\mathbb{N}$ be the ordered sets defined by $K_1:=(1,2),$
$$K_n:=(x_1^n,x_2^n,\cdots,x_{2^n}^n),$$
$$K_{n+1}:=(x_1^{n+1},x_2^{n+1},\cdots,x_{2^{n+1}}^{n+1}),$$
where $x_{2i-1}^{n+1}=(x_i^n1),$ $x_{2i}^{n+1}=(x_i^n2) $ if $i$ is odd and $x_{2i-1}^{n+1}=(x_i^n2),$ $x_{2i}^{n+1}=(x_i^n1) $ if $i$ is even. For every $k\in\mathbb{N},$ if  $w,v\in \Sigma_{2}^k,$ then $v\prec w$ iff $v=x_i^k$ and $w=x_j^k$ with $i<j.$
\end{lemma}

\begin{proof}
We will prove each direction of the equivalence by induction in $k.$ For the implication to the right, the base case is $v,w\in\Sigma_2^1$ with  $v\prec w$ iff $v=1,w=2,$ as $K_1=(1,2)=(x^1_1,x^1_2)$ we have that $v=x^1_1,w=x_2^1.$ Assume that the property is true for $n,$ and let $v,w\in\Sigma_2^{n+1}$ with $v\prec w,$ then there are three options:
\begin{enumerate}
\item $v=ax,w=by$ where $a,b\in\Sigma_2^{n},$ $x,y\in\{1,2\}$ and $a\prec b,$ or 
\item $a=b, x=1,y=2$ and the numbers of $2$ that appears in $a$ is even, or 
\item $a=b, x=2,y=1$ and the number of $2$ that appears in $a$ is odd.
\end{enumerate}
In the first case, we have the implication easily from the definition of $K_{n+1}.$ For the second and third case is enough to observe that in $K_n$ at odd coordinates the numbers of $2$ that appears is even, and at even coordinates the numbers of $2$ that appears is odd. This follows by induction in $n,$ indeed, at the even coordinate $j=2i$ of $K_{n+1}$ we have that $x_{j}^{n+1}=x_i^n 2$ if $i$ is odd or 
$x_{j}^{n+1}=x_i^n 1$ if $i$ is even, by inductive hypothesis, if $j=2i$ and $i$ odd, then the numbers of $2$ that appears in $x_i^n$ is even and therefore the numbers of $2$ that appears in $x_{j}^{n+1}=x_i^n 2$ is odd, otherwise, $j=2i$ and $i$ even, then the numbers of $2$ that appears in $x_i^n$ is odd and therefore the numbers of $2$ that appears in $x_{j}^{n+1}=x_i^n 1$ is odd. On the other hand, if $j=2i-1$ and $i$ odd, then the numbers of $2$ that appears in $x_i^n$ is even and therefore the numbers of $2$ that appears in $x_{j}^{n+1}=x_i^n 1$ is even, otherwise, $j=2i-1$ and $i$ even, then the numbers of $2$ that appears in $x_i^n$ is odd and therefore the numbers of $2$ that appears in $x_{j}^{n+1}=x_i^n 2$ is even. 

For the implication to the left, the base case is $v=1,w=2,$ then trivialy $v\prec w.$ Assume that the property is true for $n,$ and let  $v=x_i^{n+1}$ and $w=x_j^{n+1}$ with $i<j.$ We have that $x_{i}^{n+1}= x_{2k-1}^{n+1}=(x_k^n1)$ or $x_{i}^{n+1}=x_{2k}^{n+1}=(x_k^n2) $ if $k$ is odd and $x_{i}^{n+1}=x_{2k-1}^{n+1}=(x_k^n2)$ or $x_{i}^{n+1}=x_{2k}^{n+1}=(x_k^n1) $ if $k$ is even. Similarly, $x_{j}^{n+1}= x_{2l-1}^{n+1}=(x_l^n1),$ $x_{j}^{n+1}=x_{2l}^{n+1}=(x_l^n2) $ if $l$ is odd and $x_{j}^{n+1}=x_{2l-1}^{n+1}=(x_l^n2),$ $x_{j}^{n+1}=x_{2l}^{n+1}=(x_l^n1) $ if $l$ is even. There are two possibilities $l=k$ or $k<l.$ If $k<l,$ by induction in $n$ we have that $x_k^{n}\prec x_l^{n},$ then $\max_{\prec}\{x_k^{n}1,x_k^{n}2\}\prec \min_{\prec}\{x_l^{n}1,x_l^{n}2\},$ in particular $x_i^{n+1}\prec x_j^{n+1}.$ If $k=l,$ then there are four possibilities: $i=2k-1$ or $i=2k$ with $k=l$ odd, or $i=2k-1$ or $i=2k$ with $k=l$ even. If $i=2k-1$ with $k$ odd, then $j=2k,$ and we have seen that in this case  $v=(x_k^n1) \prec w=(x_k^n2).$ If $i=2k$ with $k$ odd, we have a contradiction with $j>i$ and $k=l.$ If $i=2k-1$ with $k=l$ even, then $j=2k,$ and we have in this case that $v=(x_k^n2) \prec w=(x_k^n1).$ If $i=2k$ with $k=l$ even, we have a contradiction with $j>i$ and $k=l.$ This finished the proof.
 \end{proof}

\begin{lemma}\label{pres_order}
Let $w,v\in \Sigma_{2}^*.$ If $w$ is prefix of $v,$ then $f^r_{v}[0,1]\subset f^r_{w}[0,1].$ If $w$ is not prefix of $v,$ then $\sup \left( f^r_{v}[0,1]\right)< \inf \left(f^r_{w}[0,1]\right)$ iff $v\prec w.$
\end{lemma}

\begin{proof}
Let $a,b\in \Sigma_{2}^*.$ If $a$ is prefix of $b,$ then $g^r_{b}[0,1]\subset g^r_{a}[0,1]$ follows from the contractive property of $g^r.$ Assume now that $a$ is not prefix of $b.$ We observe that $g^r_{a}$ is monotone increasing iff the numbers of $2$ that appears in $a$ is even (or zero), on the other hand, $f^r_{a}$ is monotone decreasing iff the numbers of $2$ that appears in $a$ is odd. In particular, $\sup \left( f^r_{a}[0,1]\right)=\max \left( f^r_{a}[0,1]\right)=g^r_{a}(1)$ iff the numbers of $2$ that appears in $a$ is even (or zero), and $\sup \left( f^r_{a}[0,1]\right)=\max \left( f^r_{a}[0,1]\right)=g^r_{a}(0)$ iff the numbers of $2$ that appears in $a$ is odd. Analogously, for $b.$ We have by definition that there exists words $u, v, w$ (possibly empty) such that $a = u1v,$ $b = u2w,$ and the number of $2$ that appears in $u$ is even (or zero) or $a = u2v,$ $b = u1w,$ and the number of $2$ that appears in $u$ is odd. First, assume that $a = u1v,$ $b = u2w,$ and the number of $2$ that appears in $u$ is even (or zero). We have that $\sup \left( f^r_{a}[0,1]\right)=\sup \left( f^r_{u}\circ f^r_1 \circ f^r_{v} [0,1]\right)\leq \sup \left( f^r_{u}\circ f^r_1 [0,1]\right)=f^r_{u}\circ f^r_1 (1).$ We have that $\inf (f^r_2\circ f^r_{w}[0,1])\geq \inf (f^r_2[0,1])= 1-\frac{1}{r}>\frac{1}{r}=f^r_1 (1)$ for every $r\in(2,\infty),$ therefore, by monotonicity of $f^r_{u}$ we have that $f^r_{u}\circ f^r_1 (1)\leq  f^r_{u}\circ f^r_2\circ f^r_{w}(x)=f^r_{b}(x)$ for every $x\in [0,1],$ this implies that $\sup \left( f^r_{a}[0,1]\right)\leq \inf \left( f^r_{b}[0,1]\right).$ Second, assume that $a = u2v,$ $b = u1w,$ and the number of $2$ that appears in $u$ is odd. We have that $\sup \left( f^r_{a}[0,1]\right)=\sup \left( f^r_{u}\circ f^r_2 \circ f^r_{v} [0,1]\right)\leq \sup \left( f^r_{u}\circ f^r_2 [0,1]\right)=f^r_{u}\circ f^r_2 (1).$ We have that $f^r_{u}$ is monotone decreasing and $f^r_2(1)> \sup f^r_1([0,1]),$ then $f^r_{u}\circ f^r_2 (1)< f^r_{u}\circ f^r_1 (x)$ for every $x\in [0,1],$ in particular, we have that $\sup \left( f^r_{a}[0,1]\right)\leq f^r_{u}\circ f^r_2 (1)\leq \inf f^r_{u}\circ f^r_1 \circ f^r_{w} [0,1]=\inf f^r_{b}[0,1],$ which finished the proof.
\end{proof}

The second lemma relates the previous result with the stationary probability measure $\mu^{(f^r,p)}.$

\begin{lemma} 
If $v\in \Sigma_{2}^*,$ then  $\mu^{(f^r,p)}(f^r_{v} [0,1])=\prod_{i=1}^{|v|}p_{v_i}.$
\end{lemma}

\begin{proof}
Let $v=(v_1,\ldots,v_k)\in \Sigma_{2}^*,$ then  $\mu^{(f^r,p)}(f^r_{v} [0,1])=p_{v_1} \mu^{(f^r,p)}(f^r_{(v_2,\ldots,v_k)} [0,1])=p_{v_1}p_{v_2} \mu^{(f^r,p)}(f^r_{(v_3,\ldots,v_k)} [0,1])=\cdots = p_{v_1}\cdots p_{v_k} \mu^{(f^r,p)}([0,1])=\prod_{i=1}^{|v|}p_{v_i}.$
\end{proof}

\begin{lemma} 
For every $x\in[\frac{1}{r},1],$ $p F_{r,p}(r^{-1}x)= F_{r,p}(x).$
\end{lemma}

\begin{proof}
Let $x\in [\frac{1}{r^2},\frac{1}{r}],$ then $\mu^{(f^r,p)}[0,x]=p^2+\sum_{w\in\mathcal{S}} \mu^{(f^r,p)}\left( f^r_{(12w)}[0,1]\right),$ for certain subset $\mathcal{S}\subset\Sigma_2^*.$ If we multiply both sides of the last identity by $p^{-1},$ we obtain  $p^{-1}\mu^{(f^r,p)}[0,x]=p+\sum_{w\in\mathcal{S}} \mu^{(f^r,p)}\left( f^r_{(2w)}[0,1]\right)=\mu^{(f^r,p)}[0,rx],$ where $rx\in[\frac{1}{r},1].$
\end{proof}

\begin{lemma}
Let $\mathcal{F}_{r,p}^1:=\{(x,F_{r,p}(x))^t: x\in [\frac{1}{r^2},\frac{1}{r}]\}$ and $\mathcal{F}_{r,p}^0:=\{(x,F_{r,p}(x))^t: x\in [\frac{1}{r},1]\}.$ Then $ \mathcal{F}_{r,p}^1= \begin{pmatrix} 1/r & 0 \\ 0 & p \end{pmatrix} \mathcal{F}_{r,p}^0.$ Moreover, if $n\in\mathbb{N}$ and $\mathcal{F}_{r,p}^n:=\{(x,F_{r,p}(x))^t: x\in [\frac{1}{r^{n+1}},\frac{1}{r^{n}}]\},$ then $ \mathcal{F}_{r,p}^n= \begin{pmatrix} 1/r & 0 \\ 0 & p \end{pmatrix}^{n}\mathcal{F}_{r,p}^0.$
\end{lemma}

\begin{proof}
The proof is direct from the previous lemma and induction.
\end{proof}

Let $p\in (0,1), r\in (2,\infty)$ such that $pr\geq 1$ and define $U_{r,p}(x):=x^{-\log_{r}(p)}.$

\begin{lemma}
For every $n\in\mathbb{N}_0$ define $\mathcal{U}_{r,p}^n:=\{(x,U_{r,p}(x))^t: x\in [\frac{1}{r^{n+1}},\frac{1}{r^n}]\}.$ Then $ \mathcal{U}_{r,p}^n= \begin{pmatrix} 1/r & 0 \\ 0 & p \end{pmatrix}^{n}\mathcal{U}_{r,p}^0.$
\end{lemma}

\begin{proof}
The proof follows from induction. For the base case we have that $\begin{pmatrix} 1/r & 0 \\ 0 & p \end{pmatrix}\mathcal{U}_{r,p}^0=\{(x/r,p x^{-\log_{r}(p)})^t: x\in [\frac{1}{r},1]\}=\{(x,p(r x)^{-\log_{r}(p)})^t: x\in [\frac{1}{r^2},\frac{1}{r}]\}=\{(x,x^{-\log_{r}(p)})^t: x\in [\frac{1}{r^2},\frac{1}{r}]\}=\mathcal{U}_{r,p}^1.$ Assume the case $n,$ then $ \begin{pmatrix} 1/r & 0 \\ 0 & p \end{pmatrix}^{n+1}\mathcal{U}_{r,p}^0= \begin{pmatrix} 1/r & 0 \\ 0 & p \end{pmatrix}\mathcal{U}_{r,p}^n=\begin{pmatrix} 1/r & 0 \\ 0 & p \end{pmatrix} \{(x,x^{-\log_{r}(p)})^t: x\in [\frac{1}{r^{n+1}},\frac{1}{r^n}]\}=\{(x/r,p x^{-\log_{r}(p)})^t: x\in [\frac{1}{r^{n+1}},\frac{1}{r^n}]\}=\{(x,p(r x)^{-\log_{r}(p)})^t: x\in [\frac{1}{r^{n+2}},\frac{1}{r^{n+1}}]\}=\{(x,x^{-\log_{r}(p)})^t: x\in [\frac{1}{r^{n+2}},\frac{1}{r^{n+1}}]\}=\mathcal{U}_{r,p}^{n+1}.$
\end{proof}

\begin{lemma} 
Let $p\in (0,1), r\in (2,\infty)$ such that $pr\geq 1,$ and $a_0:=1-r^{-1}+r^{-2}.$ Then $U_{r,p}(a_0)>F_{r,p}(a_0)$ if $pr> 1,$ and $U_{r,p}(a_0)=F_{r,p}(a_0)$ if $pr= 1.$
\end{lemma}

\begin{proof}  
Given $p\in(0,1),$ define the continuous functions $\psi(x)=1-x^{-1}+x^{-2}$ and $\varphi_p(x)=(\psi(x))^{-\log_x(p)}$ for $x>1.$ Then 
$$\varphi_p'(x)=\varphi_p(x)\log_x(p) \left( \frac{ x^{-2}( 2 x^{-1}-1)}{\psi(x) } + \frac{\log(\psi(x)) }{x\log(x)}\right)>0$$ for $x>1,$ $\varphi_p(1/p)=p+(1-p)^2,$ $f_p(y)<p+(1-p)^2$ for $\epsilon>0$ small enough and $1<y<1+\epsilon.$ Then $\varphi_p(x)-(p+(1-p)^2)$ has a unique root at $x=\frac{1}{p},$ $\varphi_p(x)-(p+(1-p)^2)>0$ if $x>\frac{1}{p}$ and $\varphi_p(x)-(p+(1-p)^2)<0$ if $x\in (1, \frac{1}{p}).$ In particular, $U_{r,p}(a_0)>F_{r,p}(a_0)$ if $pr> 1$ and $U_{r,p}(a_0)=F_{r,p}(a_0)$ if $pr= 1.$

\end{proof}  

We will use some properties of the function $G_{r,p}(x):=\mu^{(g^r,p)}[0,x].$ We will borrow some definitions from \cite{Kukushkin} and extends their results to our setting (observe that they considered only the case $p=1/2$). For this, observe that for $q:=r^{-1},$ the function $G_{r,p}$ takes the value $p$ at the interval $(q,1-q),$ takes the value $p^2$ at the interval $(q^2,q-q^2),$ the value $1-(1-p)^2$ at the interval $(1-q+q^2,1-q),$ etc...We will generalise this observation. Given a number $s\in(0,1),$ we say that a number $x\in(0, 1)$ has $s$-representation $x_s$ of rank $n$ if there exists a sequence $(\epsilon_i)_{i=1}^{n-1}$ with $\epsilon_i\in\{0,1\}$ for every $i,$ such that $$
x=x_s:=
\begin{cases}
a_s(x) & \mbox{ if } x<s\\
b_{1-s}(x) &\mbox{ if } x>s\\
\end{cases}
$$
where $$a_s(x):=s^n+\xi_{n}(1-s)(\epsilon_1 s^{n-2} +\epsilon_2s^{n-3}+\cdots +\epsilon_{n-1}),$$ $$b_s(x):=(1-s)( \epsilon_1+\epsilon_2 s +\cdots \epsilon_{n-1}s^{n-2} +s^{n-1}),$$ $\xi_{n}=0$ if $n=1$ and $\xi_{n}=1$ if $n>1.$ We denote by $Q_s$ the set of elements of $(0, 1)$ with $s$-representation of rank $n$ for some $n\in\mathbb{N}.$ The function $G_{r,p}$ takes the value $x=x_p\in Q_p$ at the interval $(a_{r^{-1}}(x),b_{r^{-1}}(x)).$

\begin{lemma}\label{lem_CantorMap_bound}
Let $p\in (0,1), r\in (2,\infty)$ such that $\min\{p,1-p\}r\geq 1.$ Then for every $x\in [0,1]$ $$\left(\frac{x}{r-1}\right)^{-\log_r(p)}\leq G_{r,p}(x)\leq x^{-\log_{r}(p)}.$$
\end{lemma}

The third lemma gives bounds for the Cantor map $F_{r,p}.$

\begin{lemma}\label{lem_MyMap_bound}
Let $p\in (0,1), r\in (2,\infty)$ such that $\min\{p,1-p\}r\geq 1.$ Then for every $x\in [0,1]$ $$\left(\frac{x}{r-1}\right)^{-\log_r(p)}\leq F_{r,p}(x)\leq x^{-\log_{r}(p)}.$$ Moreover, $G_{r,p}\leq F_{r,p}$ if $p\in(0,1/2),$ $G_{r,p}= F_{r,p}$ if $p =1/2,$ and $G_{r,p}\geq F_{r,p}$ if $p\in(1/2,1).$
\end{lemma}

\begin{proof}
We observe that $G_{r,p}= F_{r,p}$ if $p =1/2,$ because the stationary measure does not distinguish $1$'s and $2$'s. Let define $q=r^{-1}.$ By the symmetric properties of the graphs of $F_{r,p}$ and $G_{r,p},$ in order to prove that $G_{r,p}\leq F_{r,p}$ is enough to prove that $G_{r,p} <  F_{r,p} $ in the interval $(1-q+q^2,1-q^2),$ and to prove that $G_{r,p}\geq F_{r,p}$ is enough to prove that $G_{r,p} >  F_{r,p} $ in the interval $(1-q+q^2,1-q^2).$ 
We have that for every $x\in (1-q+q^2,1-q^2)$ $G_{r,p}(x)= 1-(1-p)^2$ and $F_{r,p}(x)= 1-(1-p)p,$ then $G_{r,p}(x)>F_{r,p}(x)$ in the interval $(1-q+q^2,1-q^2)$ iff $0< (2p-1)(p-1)$ iff $p\in(0,1/2).$ This also proves the upper bound for $F_{r,p}$ in the case $p\in(0,1/2), \min\{p,1-p\}r\geq 1,$  and the lower bound for $F_{r,p}$ in the case $p\in(1/2,1),\min\{p,1-p\}r\geq 1.$ We will prove the upper bound in the case $p\in(1/2,1),pr\geq 1,$ the lower bound in the case $p\in(0,1/2),$ $pr\geq 1,$ can be proved similarly. Let $p\in(1/2,1),$ then by the symmetric properties of $F_{r,p}(x)$  and $x^{-\log_{r}(p)}$ it is enough to prove that $F_{r,p}(x)\leq x^{-\log_{r}(p)}$ in the interval $(1-q+q^2,1-q^2).$  In this interval $F_{r,p}(x)= 1-(1-p)p,$ therefore it is enough to prove that  
\begin{equation}\label{num_11_may}
1-p+p^2\leq \inf \{ x^{-\log_{r}(p)}:x\in (1-q+q^2,1-q^2)\}=(1-q+q^2)^{-\log_{1/q}(p)}.
\end{equation}
For each $q\in (0,1/2)$, we can define the functions $f(x):=1-x+x^2$ and $g_q(x):=(1-q+q^2)^{-log(x,1/q)}.$ The function $f$ is convex and the function $g_q$ is concave, $f(1)=g_p(1),$ $f(0)=1$ and $\lim_{x\to 0^+}g_p(x)=0,$ then they intersect at exactly one point $x_q\in (0,1)$ and $f(x)>g_q(x)$ for $x<x_q,$ $f(x)<g_q(x)$ for $x>x_q.$ It is easy to prove that $x_q=q,$ then the inequality (\ref{num_11_may}) is satisfied iff $p\geq q,$ i.e.  $pr\geq 1.$
\end{proof}

\begin{remark}
We observe that in Lemma \ref{lem_CantorMap_bound}, the inequality $(1-p)r\geq 1$ is necessary for the upper bound and the inequality $pr\geq 1$ is necessary for the lower bound. While, in Lemma \ref{lem_MyMap_bound} the inequality $(1-p)r\geq 1$ is necessary for the lower bound and the inequality $pr\geq 1$ is necessary for the upper bound.
\end{remark}

\begin{lemma}
Let $n,k,l\in \mathbb{N}$ and call $p:=\frac{1}{2k+1}\in (0,\frac{1}{3}).$ Then the equation 
\begin{equation}\label{eq:16may}
p^n=\left(1-p\right)^l \sum_{i=1}^m\sum_{(a,b)\in A_i}p^{a}(1-p)^b
\end{equation}
does not have solution for $m\in \mathbb{N}$ and $A_i\subset \mathbb{N}_0\times \mathbb{N}_0$ with $A_i$ a finite set for $i\in\{1,\ldots,m\}.$
\end{lemma}

\begin{proof}
Let $n,k,l,p$ as in the statement of the lemma. Suppose the equation (\ref{eq:16may}) has solution for some $m\in \mathbb{N}$ and $A_i\subset \mathbb{N}_0\times \mathbb{N}_0$ with $A_i$ a finite set for $i\in\{1,\ldots,m\}.$ Define $$r:=\max_{i\in\{1,\ldots,m\}}\max_{(a,b)\in A_i} a+b.$$ It is well defined as all $A_i$ are finite. Then multiplying at both sides of the equation (\ref{eq:16may}) by $(2k+1)^{n+l+r}$ we obtain
$$
(2k+1)^{l+r}=\left(2k\right)^l \sum_{i=1}^m\sum_{(a,b)\in A_i} (2k+1)^{n+r-(a+b)}(2k)^{b},
$$
this is a contradiction, because the left part of the equation is an odd number while the right part is even.
\end{proof}

A direct consequence

\begin{lemma}
Let $p:=\frac{1}{2k+1}$ for $k\in\mathbb{N}.$ Then the equation 
\begin{equation}\label{eq:16may}
p^n=\sum_{i=1}^m\sum_{(a,b)\in A_i}p^{a}(1-p)^b
\end{equation}
does not have solution for $n\in\mathbb{N},$ $m\in \mathbb{N}$ and $A_i\subset  \mathbb{N}\times \mathbb{N}$ with $A_i$ a finite set for $i\in\{1,\ldots,m\}.$
\end{lemma}

The proof of these lemmas is essentially the same than the one for the following that we will use

\begin{lemma}\label{lem_needed}
Let $n,k,l\in \mathbb{N}$ and call $p:=\frac{1}{2k+1}\in (0,\frac{1}{3}).$ Then the equation 
\begin{equation}\label{eq:16may_second}
p^n=(1-p)^n+\left(1-p\right)^l \sum_{i=1}^m\sum_{(a,b)\in A_i}(p^{a}(1-p)^b-p^{b}(1-p)^a)
\end{equation}
does not have solution for $m\in \mathbb{N}$ and $A_i\subset \mathbb{N}_0\times \mathbb{N}_0$ with $A_i$ a finite set for $i\in\{1,\ldots,m\}.$
\end{lemma}

\begin{proof}
Let $n,k,l,p$ as in the statement of the lemma. Suppose the equation (\ref{eq:16may_second}) has solution for some $m\in \mathbb{N}$ and $A_i\subset \mathbb{N}_0\times \mathbb{N}_0$ with $A_i$ a finite set for $i\in\{1,\ldots,m\}.$ Define $$r:=\max_{i\in\{1,\ldots,m\}}\max_{(a,b)\in A_i} a+b$$ and multiplying at both sides of the equation (\ref{eq:16may_second}) by $(2k+1)^{n+l+r}$ we obtain
$$
(2k+1)^{l+r}=\left(2k\right)^{n}\left(2k+1\right)^{l+r} +\left(2k\right)^l \sum_{i=1}^m\sum_{(a,b)\in A_i} (2k+1)^{n+r-(a+b)}((2k)^{b}-(2k)^{a}),
$$
this is a contradiction, because the left part of the equation is an odd number while the right part is even.
\end{proof} 

A direct consequence

\begin{lemma}
Let $p:=\frac{1}{2k+1}$ for $k\in\mathbb{N}.$ Then the equation 
\begin{equation}\label{eq:16may}
p^n=\sum_{i=1}^m\sum_{(a,b)\in A_i}(p^{a}(1-p)^b-p^{b}(1-p)^a)
\end{equation}
does not have solution for $n\in\mathbb{N},$ $m\in \mathbb{N}$ and $A_i\subset  \mathbb{N}\times \mathbb{N}$ with $A_i$ a finite set for $i\in\{1,\ldots,m\}.$
\end{lemma}

Given an element $w\in\Sigma_2^*$ we define $\#1(w)$ equal to the numbers of $1$ that appears in $w$ and $\#2(w)$ the numbers of $2$ that appears in $w.$ Given $p\in (0,1),w\in\Sigma_2^*$ we define $p_w=p^{\#1(w)}(1-p)^{\#2(w)}.$ It is easy from the definition of the function $F_{r,p}$ to observe that for $w\in \Sigma_2^n$ 
$$
F_{r,p}(\sup f_w[0,1])=\sum_{v\in \Sigma_2^n: v\prec w} p_v.
$$

Moreover, recalling the definition of $K_n=(x_1^n,x_2^n,\cdots,x_{2^n}^n)$ in Lemma \ref{lem_de_los_K}, we have a characterisation given by the following lemma

\begin{lemma}
Let $p\in(0,1),n\in \mathbb{N}$ and $i\in \{1,\ldots,2^n\}.$ Then
$$F_{r,p}(\sup f_{x_i^n}[0,1])=\sum_{j=1,\ldots,i} p_{x_j^n}.$$
\end{lemma}

This lemma implies that if $F_{r,p}$ and $F_{r,1-p}$ coincide at certain point, then necessarily there must exists $n\in\N$ and $i\in\{1,\ldots,2^n\}$ such that
\begin{equation}\label{eq_goal_today}
\sum_{j=1,\ldots,i} p_{x_j^n}=\sum_{j=1,\ldots,i} (1-p)_{x_j^n}.
\end{equation}

We will prove that this equation does not have solution under certain conditions, before we obtain a relationship between the sets $K_n=(x_1^n,x_2^n,\cdots,x_{2^n}^n)$ with the numbers $p_{x_i^n}.$ 

\begin{lemma}\label{lem_induction}
For any $n\in\mathbb{N}\setminus\{1,2\},$ we have the following:
\begin{enumerate}
\item $\{p_{x_{1}^n},\ldots,p_{x^n_{2^{n-1}-2^{n-3}}}\}=\{(1-p)_{x_{2^{n-1}-2^{n-3}+1}^n},\ldots, (1-p)_{x_{2^{n}-2^{n-2}}^n}\}.$
\item $\{(1-p)_{x_{1}^n},\ldots,(1-p)_{x^n_{2^{n-1}-2^{n-3}}}\}=\{p_{x_{2^{n-1}-2^{n-3}+1}^n},\ldots, p_{x_{2^{n}-2^{n-2}}^n}\}.$
\item $p_{x_{1}}^n=p^n.$
\item $p_{x_{s_n}}^n=(1-p)^n$ for a unique $s_n\in \{2^{n-1}-2^{n-3}+1,\ldots, 2^{n}-2^{n-2}\}.$
\item $p_{x_{j}}^n=p^a(1-p)^b$ with $a,b>0$ for every $i\in\{1,2^{n}-2^{n-2}\}\setminus\{1,s_n\}.$
\item $(1-p)_{x_{j}}^n=p^a(1-p)^b$ with $a,b>0$ for every $i\in\{1,2^{n}-2^{n-2}\}\setminus\{1,s_n\}.$
\item \label{lem_penultimo} $p_{x_{2^{n}-2^{n-2}}^n}=(1-p)^3 p^{n-3}.$
\item \label{lem_ultimo} $n\in\mathbb{N}\setminus\{1,2,3\},$ $p_{x_{i}^n}=(1-p)^3 p^{a}(1-p)^{b}$ with $n-3>a\geq 0$ and  $n-3 \geq b> 0$ for every $i\in\{s_n,2^{n}-2^{n-2}-1\}.$
\end{enumerate}

It is not hard to prove each part of the lemma by induction in $n$. We provide a proof for $\ref{lem_penultimo}$ and $\ref{lem_ultimo}.$

\begin{proof}
First, we prove \ref{lem_penultimo}. In order to prove that $p_{x_{2^{n}-2^{n-2}}^n}=(1-p)^3 p^{n-3}$ we use induction in $n$ to prove that $\#1(x_{2^{n}-2^{n-2}}^n)=n-3$ and $\#2(x_{2^{n}-2^{n-2}}^n)=3.$ For the base case $n=3$ we have that $x_{2^{3}-2^{3-2}}^3=x_{6}^3=222,$ therefore $\#1(222)=0$ and $\#2(222)=3.$ For $x_{2^{n+1}-2^{n+1-2}}^{n+1}$ we have by definition of $K_{n+1}$ that $\#1(x_{2^{n+1}-2^{n+1-2}}^{n+1})=\#1(x_{2^{n}-2^{n-2}}^{n})+1$ and $\#2(x_{2^{n+1}-2^{n+1-2}}^{n+1})=\#2(x_{2^{n}-2^{n-2}}^{n}),$ because $2^{n+1}-2^{n+1-2}$ is even. Then, by inductive hypothesis we have that $\#1(x_{2^{n+1}-2^{n+1-2}}^{n+1})=n-3+1$ and $\#2(x_{2^{n+1}-2^{n+1-2}}^{n+1})=n-3,$ which concludes the proof.

Second, we prove \ref{lem_ultimo}.  We use induction in $n$ to prove that $\#1(x_{i}^n)< n-3$ for every $i\in\{s,2^{n}-2^{n-2}-1\}.$ For the base case $n=4$ we have that $s_n=11$ and $2^{4}-2^{4-2}-1=11,$ in this case $x_{11}^4=2222,$ then $\#1(2222)=0<4-3=1.$ For the case $i\in\{s_{n+1},2^{n+1}-2^{n+1-2}-1\}$ we have by construction of $K_{n+1}$ that $\#1(x_{i}^{n+1})\in \{\#1(x_{j}^n), \#1(x_{j}^n)+1\}$ for $j\in\{s_n,2^{n}-2^{n-2}-1\}.$ By inductive hypothesis que have that $\#1(x_{j}^n)<n-3$ for every $j\in\{s_n,2^{n}-2^{n-2}-1\},$ then $\#1(x_{i}^{n+1})<n+1-3,$ which concludes the proof.
\end{proof}

\end{lemma}

\begin{lemma}\label{lema_global}
Let $p=1/q$ for $q$ an odd positive integer. Then the equation (\ref{eq_goal_today}) does have solution for $n\in\mathbb{N}\setminus\{1\}$ and $i\in\{1,\ldots,2^n-2^{n-2}-1\}.$ 
\end{lemma}

\begin{proof}
We can use Lemma \ref{lem_induction} (parts 3,4,5 and 6) to prove to prove that the equation (\ref{eq_goal_today}) does not have solution for $n\in \mathbb{N}\setminus\{1\}$ and $i\in \{1,\ldots,s_n-1\}.$ Indeed, it there was a solution, then after cancelation from the symmetric terms at both sides at both of the sums
$$
\sum_{j=1,\ldots,i} p_{x_j^n}=\sum_{j=1,\ldots,i} (1-p)_{x_j^n}
$$
we obtain 
$$
p^n+\sum_{k=1,\ldots,t} p_{x_{i_k}^n}=(1-p)^n+\sum_{k=1,\ldots,t} (1-p)_{x_{i_t}^n}
$$
where for every $k=1,\ldots,t$ we have that $p_{x_{i_k}^n}=p^{a_k} (1-p)^{b_k}$ and $(1-p)_{x_{i_t}^n}=p^{c_k} (1-p)^{d_k}$ for some $a_k,b_k,c_k,d_k>0.$ Substracting $\sum_{k=1,\ldots,t} p_{x_{i_k}^n}$ at both sides of the equation and factorising by $(1-p)$ we obtain 
$$
p^n=(1-p)^n+(1-p)\sum_{k=1,\ldots,t} (p^{c_k} (1-p)^{d_k-1}-p^{a_k} (1-p)^{b_k-1}).
$$
The last equation does not have solution by Lemma \ref{lem_needed}.

Therefore, it is left to prove that the equation does not have solution for $n\in \mathbb{N}\setminus\{1,2\}$ and $i\in \{s_n,\ldots, 2^n-2^{n-2}-1\}.$ It is easy to see that if the equation (\ref{eq_goal_today}) has solution for some $n_0\in \mathbb{N}\setminus\{1,2\}$ and $i^*\in \{s_{n_0},\ldots, 2^{n_0}-2^{n_0-2}-1\}$ then
\begin{equation}\label{la_cola}
\sum_{j=i^*+1,\ldots, 2^{n_0}-2^{n_0-2}} p_{x_j^{n_0}}=\sum_{j=i^*+1,\ldots, 2^{n_0}-2^{n_{0}-2}} (1-p)_{x_j^{n_{0}}}.
\end{equation}
We can use Lemma \ref{lem_induction} (parts 7 and 8) to prove that $p_{x_{2^{n_0}-2^{n_0-2}}^{n_0}}=(1-p)^3 p^{n_0-3}$ and $p_{x_{i}^{n_0}}=(1-p)^3 p^{a}(1-p)^{b}$ with $n_0-3>a\geq 0$ and  $n_0-3 \geq b> 0$ for every $i\in\{s_{n_0},2^{n_0}-2^{n_0-2}-1\}$ if $n_0\in \mathbb{N}\setminus\{1,2,3\},$ being the case $n_0=3$ easy to prove that the equation (\ref{la_cola}) does not have solution. We assume now that $n_0\in \mathbb{N}\setminus\{1,2,3\}$ and multiply at both sides of the equation (\ref{la_cola}) by $(1-p)^{-3}p^3,$ obtaining 
$$
p^n+\left(1-p\right) \sum_{i=1}^m\sum_{(a,b)\in A_i}p^{a}(1-p)^b=(1-p)^n+\left(1-p\right) \sum_{i=1}^m\sum_{(a,b)\in A_i}p^{a}(1-p)^b
$$
for some $m\in \mathbb{N}$ and $A_i\subset \mathbb{N}_0\times \mathbb{N}_0$ with $A_i$ a finite set for $i\in\{1,\ldots,m\}.$ Substracting  $\left(1-p\right) \sum_{i=1}^m\sum_{(a,b)\in A_i}p^{a}(1-p)^b$ at both sides we obtain the equation

$$
p^n=(1-p)^n+\left(1-p\right) \sum_{i=1}^m\sum_{(a,b)\in A_i} (p^{a}(1-p)^b- p^{b}(1-p)^a),
$$
that does not have solution in virtue of Lemma \ref{lem_needed}, therefore we obtain a contradiction, which finished the proof.
\end{proof}

The fourth lemma gives the first intersection between the graphs of the two Cantor map $F_{r,p}$ and $F_{r,1-p}.$

\begin{lemma}\label{lem_bound}
Let $p\in (0,1/2)$ such that $p=1/n$ with $n$ an odd natural, $r\in (2,\infty)$ and define $q:=1/r.$ Then $F_{r,1-p}(x)>F_{r,p}(x)$ for every $x\in (0,1-q+q^{2}),$ and $F_{r,1-p}(x)=F_{r,p}(x)$ for every $x$ in $(1-q+q^{2},1-q^{2}).$
\end{lemma}

\begin{proof}
In the interval  $(1-q+q^{2},1-q^{2})$ we have that $$F_{r,1-p}(x)=1-\left(1-(1-p)\right)(1-p)=1-p(1-p)$$ and $F_{r,p}(x)=1-(1-p)p,$ therefore $F_{r,1-p}(x)=F_{r,p}(x)$ for every $x$ in  $(1-q+q^{2},1-q^{2}).$ For the proof of the other part, we use Lemma \ref{lema_global}.
\end{proof}

\begin{remark}
It is not know for the author if the fact that $F_{r,1-p}(x)>F_{r,p}(x)$ for every $x\in (0,1-q)$ can be proved using similar ideas from the proof of Lemma 3.3 in \cite{Mark_Pollicott_Italo_Cipriano}. It is not direct that it could be done, because the order $\prec$ is more delicate than the lexicographic order.
\end{remark}

\begin{remark}
Lemma Lemma \ref{lem_MyMap_bound} does not alone imply Lemma \ref{lem_bound}, indeed, for example, it gives no information for $r=4$ and $p\in (2/5,3/5).$
\end{remark}

\begin{remark}
Lemma \ref{lem_bound} should be still true for every $p\in (0,1/2),$ however, the author does not know a proof. 
\end{remark}

The fifth lemma is a ``zooming in and re-scaling'' property of the interactions between the graphs of the two Cantor map $F_{r,p}$ and $F_{r,1-p}.$ Define the map $S:\mathbb{R}^2\to \mathbb{R}^2$ by
$$
S(x,y):= \begin{pmatrix}-q^2&0\\0&-p(1-p) \end{pmatrix} \begin{pmatrix}x\\y \end{pmatrix}+\begin{pmatrix}1\\1 \end{pmatrix}.
$$
Observe that 
$$
\begin{pmatrix}-q^2&0\\0&-p(1-p) \end{pmatrix} \begin{pmatrix}x\\y \end{pmatrix}=\begin{pmatrix}\cos \pi &0\\0&\cos \pi \end{pmatrix}  \begin{pmatrix}q^2&0\\0&p(1-p) \end{pmatrix}, 
$$
Therefore the map $S$ acts by contraction ($q^2$ in the $x$-axis and $p(1-p)$ in the $y$-asis), rotation in $\pi$ and translation.

\begin{lemma}\label{last_lemma} 
Let $p\in (0,1), r\in (2,\infty)$ such that $\min\{p,1-p\}r\geq 1.$ Then
$$
\{S(x,F_{r,p}(x)):x\in [0,1]\}=\{(x,F_{r,p}(x)):x\in [1-q^2,1]\}.
$$
\end{lemma}

\begin{proof}
The proof follows by definition of $F_{r,p}.$
\end{proof}

Define $\Pi_1:\mathbb{R}^2\to \mathbb{R}$ denotes the projection on the first coordinate $\Pi_1(x,y)=x$ and $\Pi_2:\mathbb{R}^2\to \mathbb{R}$ denotes the projection on the second coordinate $\Pi_2(x,y)=y.$

In particular, combining Lemma \ref{last_lemma} and Lemma  \ref{lem_bound} we have the following. 

\begin{lemma}
Let $p\in (0,1/2)$ such that $p=1/n$ with $n$ an odd natural, $r\in (2,\infty)$ and define $q:=1/r.$ Then the graphs of $F_{r,p}$ and $F_{r,1-p}$ coincide exactly at the intervals $(a_k,b_k)$ for $k\in\mathbb{N}_0,$ where $q^*:=1-q^{2}, q_*:=1-q+q^{2}, p_*:=1-p(1-p),$
$$
a_k=
\begin{cases}
\Pi_1S^k(q_*,p_*) &\mbox{ if  $k$ even or zero,}\\
\Pi_1S^k(q^*,p_*) &\mbox{ if  $k$ odd,}
\end{cases}
$$
and
$$
b_k=
\begin{cases}
\Pi_1S^k(q^*,p_*) &\mbox{ if  $k$ even or zero,}\\
\Pi_1S^k(q_*,p_*) &\mbox{ if  $k$ odd.}
\end{cases}
$$
Moreover, for every $x\in (a_k,b_k)$ 
$$
F_{r,p}(x)=F_{r,1-p}(x)=\Pi_2 S^k(q_*,p_*).
$$
\end{lemma}

By induction we deduce that 

\begin{lemma}
Let $p\in (0,1/2)$ such that $p=1/n$ with $n$ an odd natural and $k\in\mathbb{N}.$ Then for every $x\in (b_{2k-1},a_{2k+1})$ 
$$
F_{r,1-p}(x)>F_{r,p}(x)
$$
and for every $x\in (b_{2(k+1)},a_{2k})$
$$
F_{r,1-p}(x)<F_{r,p}(x).
$$
Also, for every $x\in (0,a_1)$
$$
F_{r,1-p}(x)>F_{r,p}(x)
$$
and for every $x\in (b_2,1)$
$$
F_{r,1-p}(x)<F_{r,p}(x).
$$
\end{lemma}

We observe that the limit $$\lim_{k\to\infty}\Pi_1 S^k(q_*,p_*)=\lim_{k\to\infty}\Pi_1 S^k(q^*,p_*)=\frac{1}{1+q^2},$$ then from the previous lemma we have the following

\begin{lemma}
Let $p\in (0,1/2)$ such that $p=1/(2k+1)$ with $k\in\mathbb{N}.$ For every $x \in (0,\frac{1}{1+q^2})$ 
$$
F_{r,1-p}(x)>F_{r,p}(x)
$$
and for every $x \in (\frac{1}{1+q^2},1)$
$$
F_{r,1-p}(x)<F_{r,p}(x).
$$
\end{lemma}

\section{Computations and examples}

We provide a computational part, including a few examples for each theorem.

\subsubsection{Examples Theorem \ref{kmaps} }

Examples where Theorem \ref{kmaps} applies:

\begin{description}
\item[E.g. 1] $f=(f_1,f_2,f_3)$ with 
$$
\begin{aligned}
f_1x&=x/5\\
f_2x&=x/5+2/5\\
f_3x&=x/5+4/5
\end{aligned}
$$
 $p=(1/2,1/4,1/4)$ and $q=(1/4,1/4,1/2).$

\item[E.g. 2] $f=(f_1,f_2,f_3)$ with 
$$
\begin{aligned}
f_1x&=x/5\\
f_2x&=3x/5+1/5\\
f_3x&=x/5+4/5
\end{aligned}
$$
 $p=(1/4,1/3,5/12)$ and $q=(1/6,1/4,7/12).$ 
\item[E.g. 3]
$$
\begin{aligned}
f_1x&=\sin(\pi x/4)/6\\
f_2x&=x/6+1/3\\
f_3x&=\sin(\pi x/4)/3+2/3
\end{aligned}
$$
 $p=(0.1,0.3,0.6)$ and $q=(0.2,0.5,0.3).$ 
\end{description}

Examples where Theorem \ref{teo_cuatro} does not apply:

\begin{description}
\item[E.g. 4] $$
\begin{aligned}
f_1x&=\sin(\pi x/4)/6\\
f_2x&=x/6+1/3\\
f_3x&=\sin(\pi x/4)/3+2/3
\end{aligned}
$$
 $p=(0.3,0.1,0.6)$ and $q=(0.2,0.5,0.3).$ 
\item[E.g. 5] $$
\begin{aligned}
f_1x&=\sin(\pi x/4)/6\\
f_2x&=-x/6+1/2\\
f_3x&=\sin(\pi x/4)/3+2/3
\end{aligned}
$$
 $p=(0.1,0.3,0.6)$ and $q=(0.2,0.5,0.3).$ 
\end{description}

\subsubsection{Examples Theorem \ref{teo_dos} }

Examples where Theorem \ref{teo_dos} applies:

\begin{description}
\item[E.g. 6] $$
\begin{aligned}
f_1x&=x/3 & g_1x&=x/6\\
f_2x&=x/3+2/3& g_2x&=x/6+2/3\\
\end{aligned}
$$
 $p=(0.4,0.6)$ and $q=(0.5,0.5).$ 
\end{description}

Examples where Theorem \ref{teo_dos} does not apply:

\begin{description}
\item[E.g. 7] $$
\begin{aligned}
f_1x&=x/3 & g_1x&=x/6\\
f_2x&=x/3+2/3& g_2x&=x/6+2/3\\
\end{aligned}
$$
 $p=(0.5,0.5)$ and $q=(0.4,0.6).$
 \item[E.g. 8] $$
\begin{aligned}
f_1x&=x/3 & g_1x&=x/6\\
f_2x&=x/3+2/3& g_2x&=x/6+2/3\\
\end{aligned}
$$
 $p=(0.5,0.5)$ and $q=(0.25,0.75).$ 
\end{description}

\subsubsection{Examples Theorem \ref{teo_cuatro} }

Examples where Theorem \ref{teo_cuatro} applies:

\begin{description}
\item[E.g. 9] $$
\begin{aligned}
f_1^3x&=x/3\\
f_2^3x&=1-x/3\\
\end{aligned}
$$
 $p=(1/3,2/3)$ and $q=(2/3,1/3).$ 
\item[E.g. 10] $$
\begin{aligned}
f_1^7x&=x/7\\
f_2^7x&=1-x/7\\
\end{aligned}
$$
 $p=(1/5,4/5)$ and $q=(4/5,1/5).$ 
\end{description}

Example where Theorem \ref{teo_cuatro} does not apply:

\begin{description}
\item[E.g. 11] $$
\begin{aligned}
f_1^3x&=x/3\\
f_2^3x&=1-x/3\\\end{aligned}
$$
 $p=(0.1,0.9)$ and $q=(0.3,0.4).$
\end{description}

Examples of the bounds obtained in Lemma \ref{lem_MyMap_bound}.

\begin{description}
\item[E.g. 12] $$
\begin{aligned}
f_1^{2.1}x&=x/2.1\\
f_2^{2.1}x&=1-x/2.1
\end{aligned}
$$
 $p=(1/2.1,1.1/2.1).$
 \item[E.g. 13] $$
\begin{aligned}
f_1^{3}x&=x/3\\
f_2^{3}x&=1-x/3
\end{aligned}
$$
 $p=(2/3,1/3).$ 
\end{description}

Examples where Proposition \ref{prop_cinco} applies:

\begin{description}
\item[E.g. 14] $$
\begin{aligned}
f_1x&=x/3 & g_1x&=x/6\\
f_2x&=x/3+2/3& g_2x&=x/6+2/3\\
\end{aligned}
$$
 $p=(0.5,0.5)$ and $q=(0.4,0.6).$
\end{description}

Examples where Proposition \ref{prop_cinco} does not apply:

\begin{description}
\item[E.g. 15] $$
\begin{aligned}
f^{2.5}_1 x =& \frac{x}{2.5} & g^{3}_1 x =& \frac{x}{3}\\
f^{2.5}_2 x =& 1-\frac{x}{2.5} & g^{3}_2 x =& \frac{x}{3}+\frac{2}{3}
\end{aligned}
$$
 $p=(0.25,0.75)$ and $q=(0.25,0.75).$
 \item[E.g. 16] $$
\begin{aligned}
f^{3}_1 x =& \frac{x}{3} & g^{3}_1 x =& \frac{x}{3}\\
f^{3}_2 x =& 1-\frac{x}{3}& g^{3}_2 x =& \frac{x}{3}+\frac{2}{3}
\end{aligned}
$$
 $p=(0.25,0.75)$ and $q=(1/3,2/3).$ 
\end{description}

\begin{figure}[!]
\centering
\includegraphics[scale=0.33]{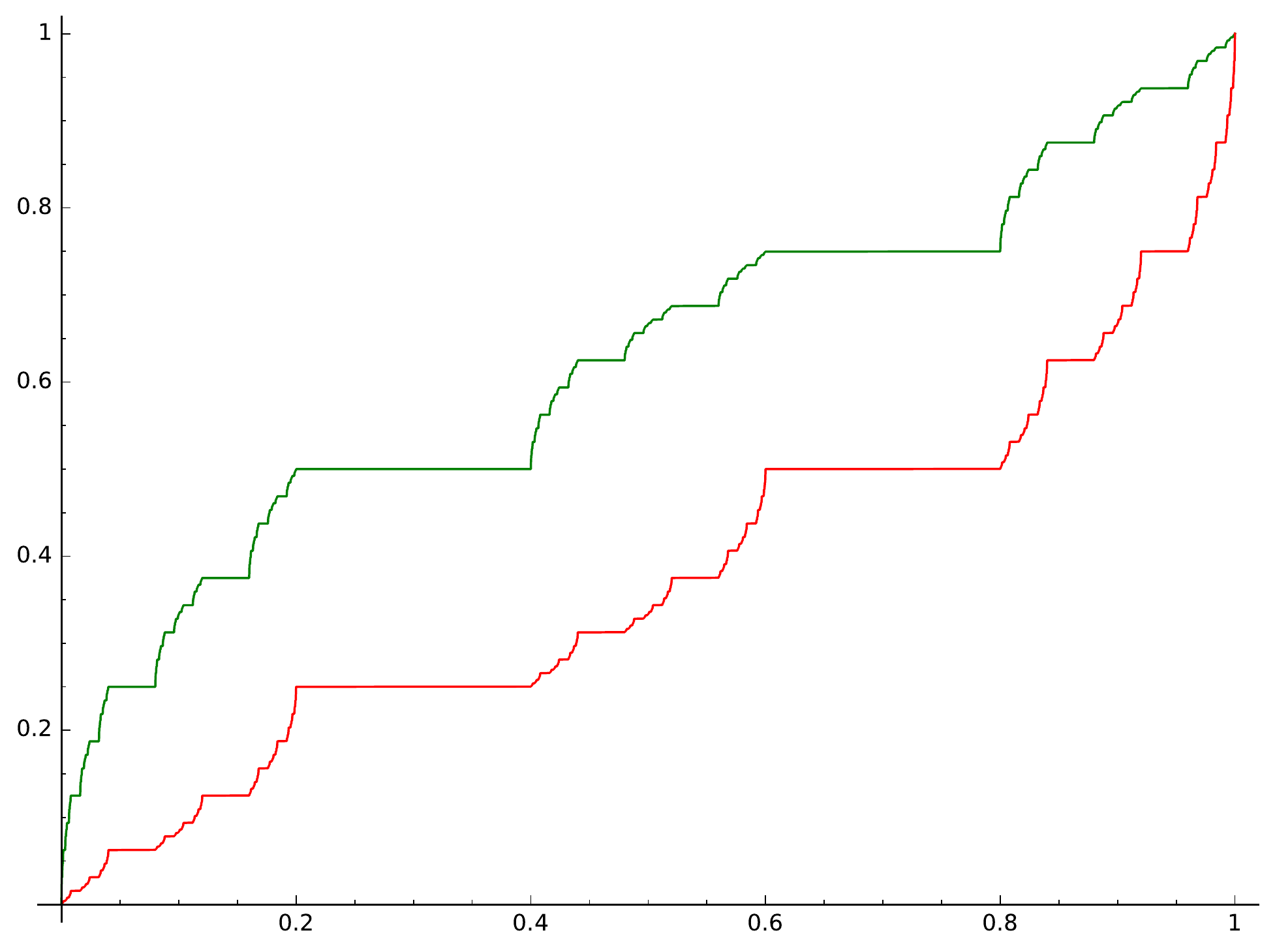}
\caption{E.g. 1. Cantor staircase for $\mu^{(f,p)}[0,x]$ in green and for $\mu^{(f,q)}[0,x]$ in red.}
\end{figure}

\begin{figure}[!]
\centering
\includegraphics[scale=0.33]{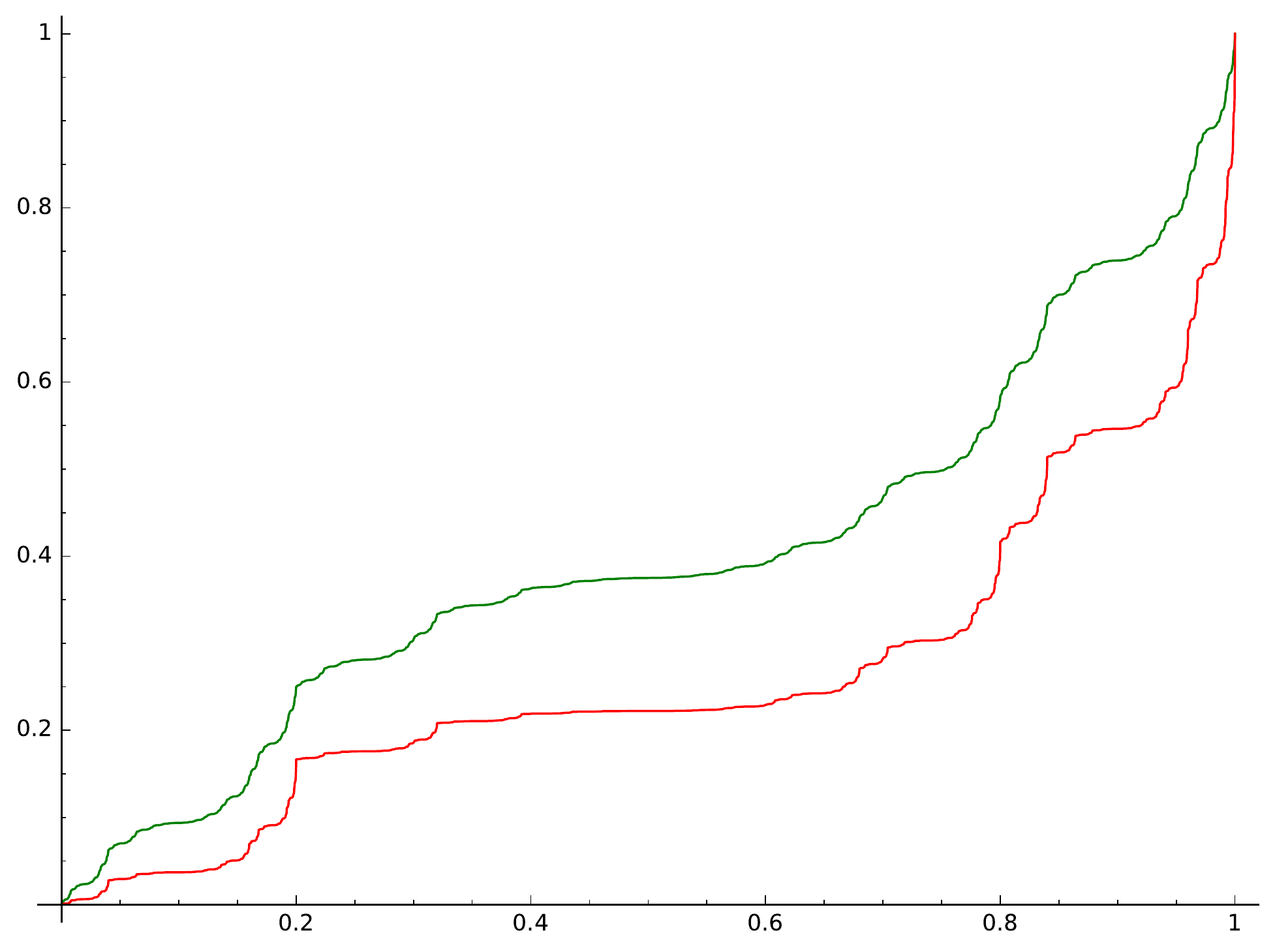}
\caption{E.g. 2. Cantor staircase for $\mu^{(f,p)}[0,x]$ in green and for $\mu^{(f,q)}[0,x]$ in red.}
\end{figure}

\begin{figure}[!]
\centering
\includegraphics[scale=0.33]{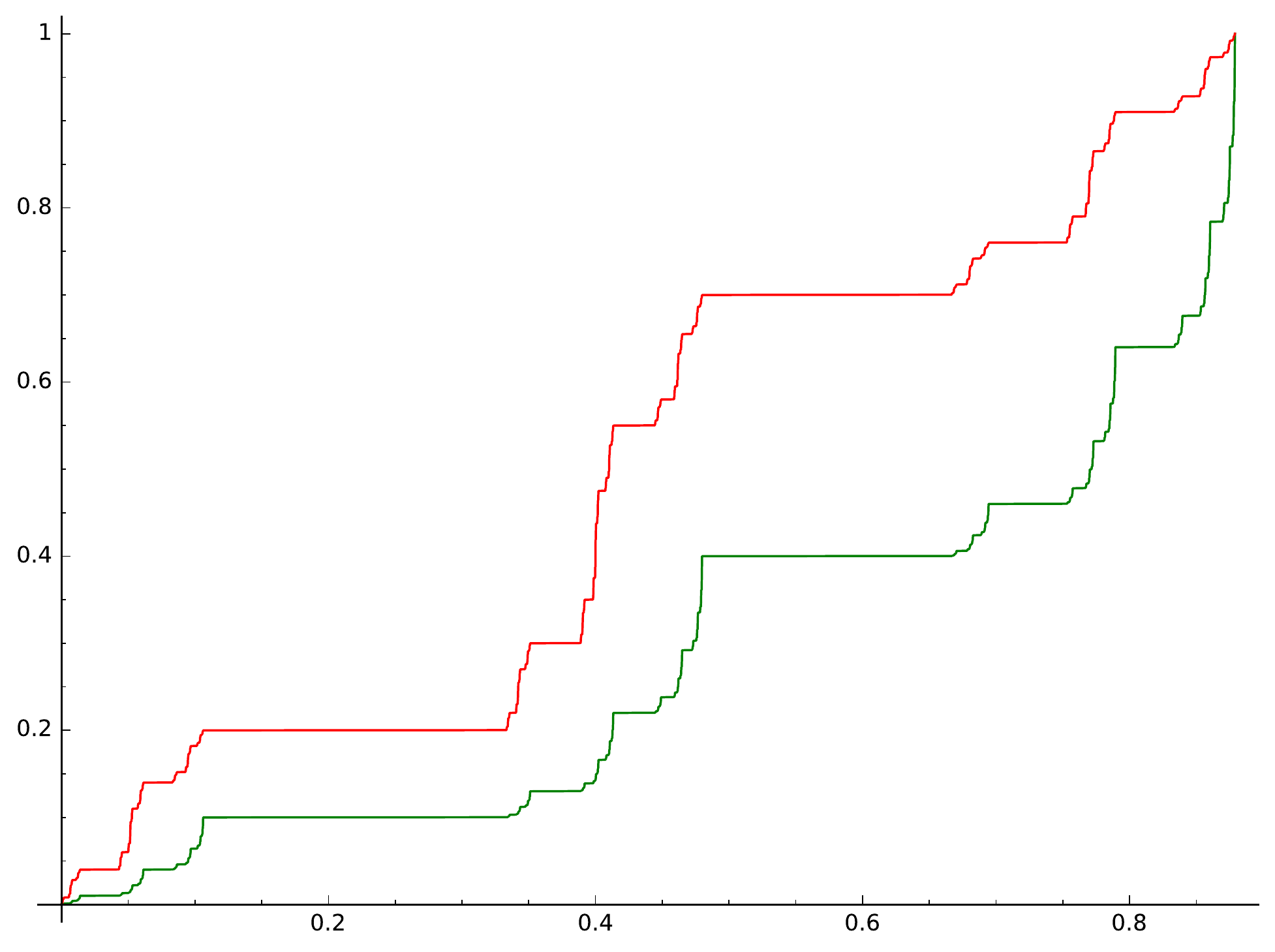}
\caption{E.g. 3. Cantor staircase for $\mu^{(f,p)}[0,x]$ in green and for $\mu^{(f,q)}[0,x]$ in red.}
\end{figure}

\begin{figure}[!]
\centering
\includegraphics[scale=0.33]{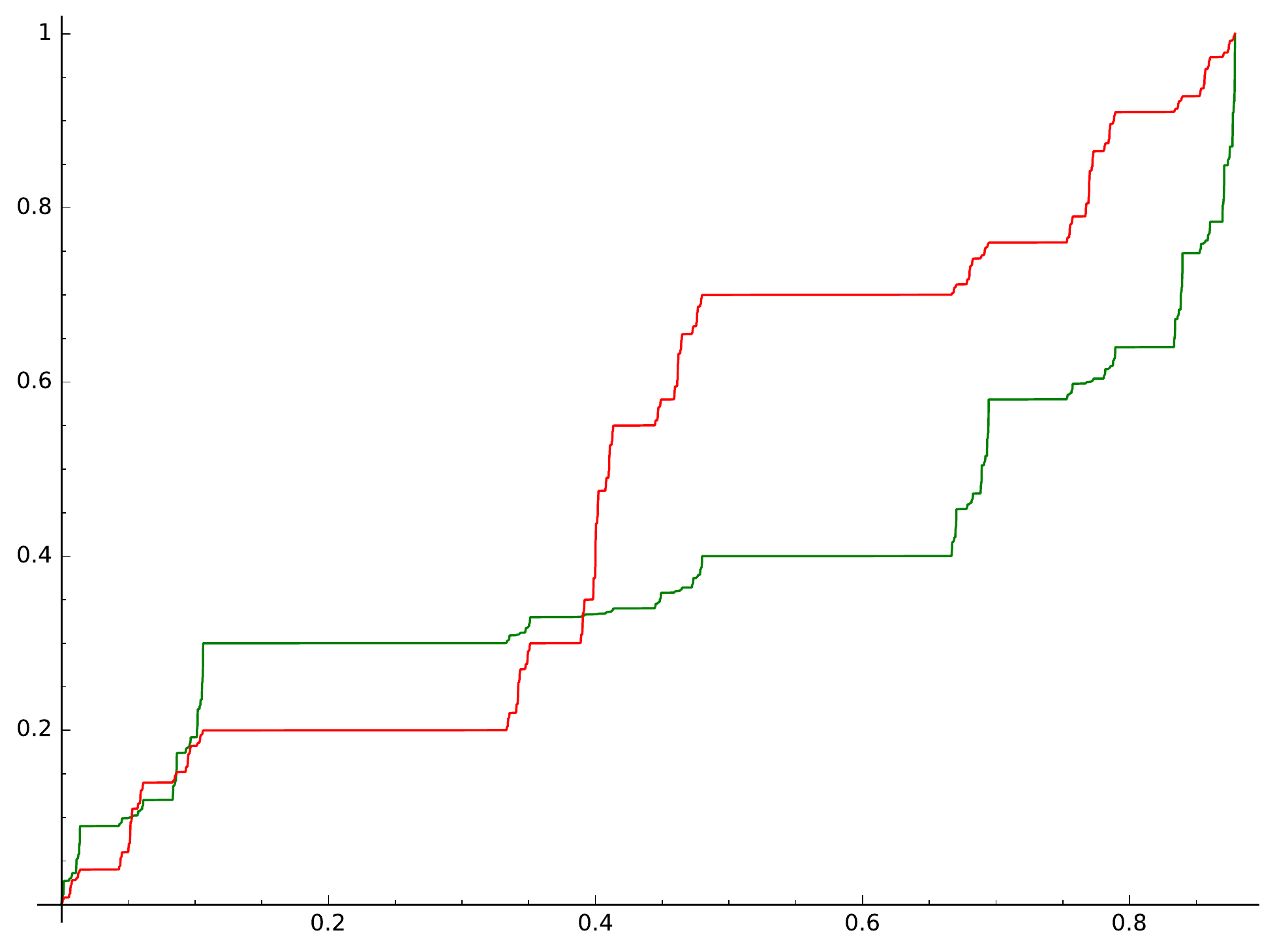}
\caption{E.g. 4. Cantor staircase for $\mu^{(f,p)}[0,x]$ in green and for $\mu^{(f,q)}[0,x]$ in red.}
\end{figure}

\begin{figure}[!]
\centering
\includegraphics[scale=0.33]{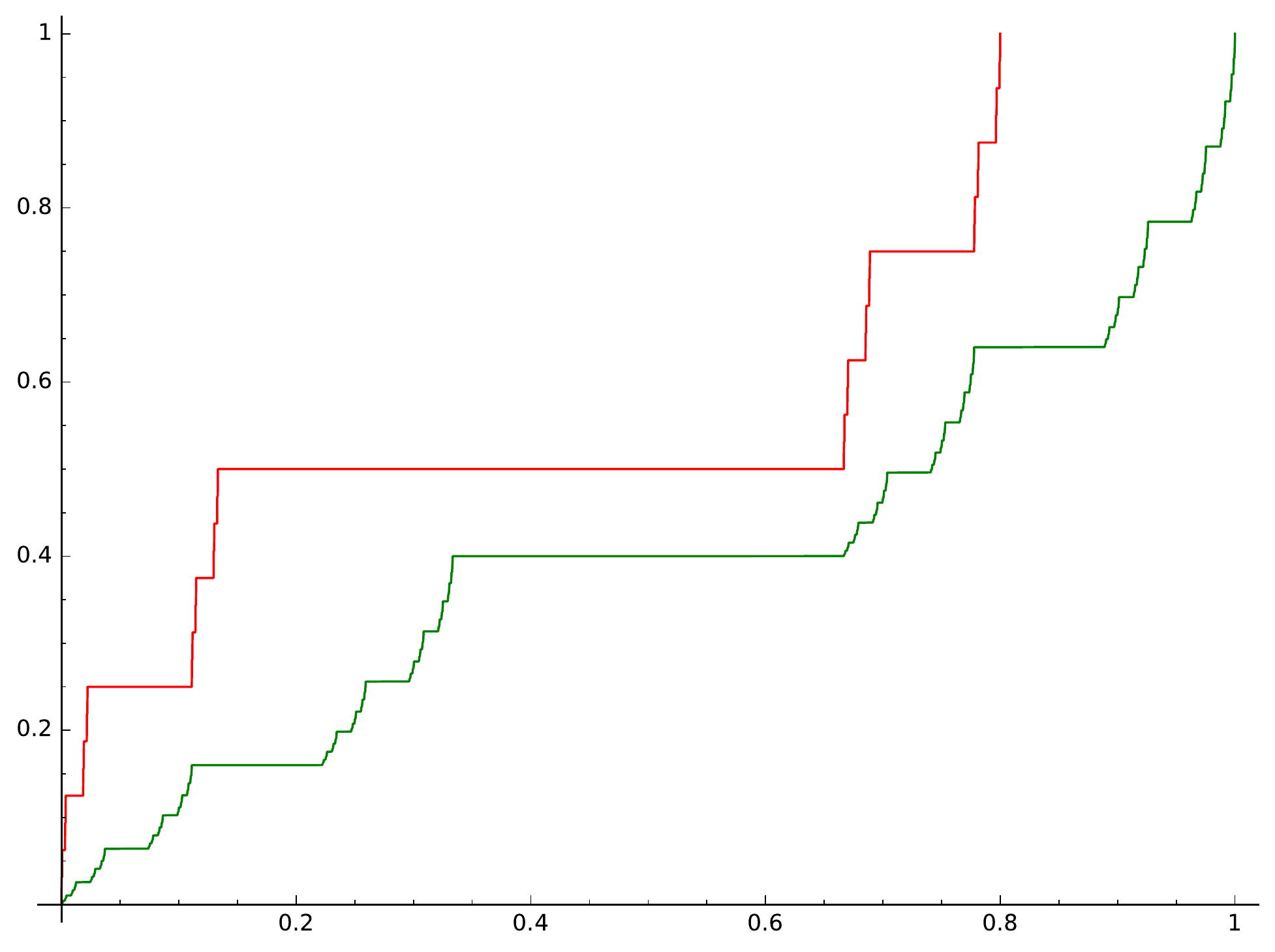}
\caption{E.g. 6. Cantor staircase for $\mu^{(f,p)}[0,x]$ in green and for $\mu^{(g,q)}[0,x]$ in red.}
\end{figure}

\begin{figure}[!]
\centering
\includegraphics[scale=0.33]{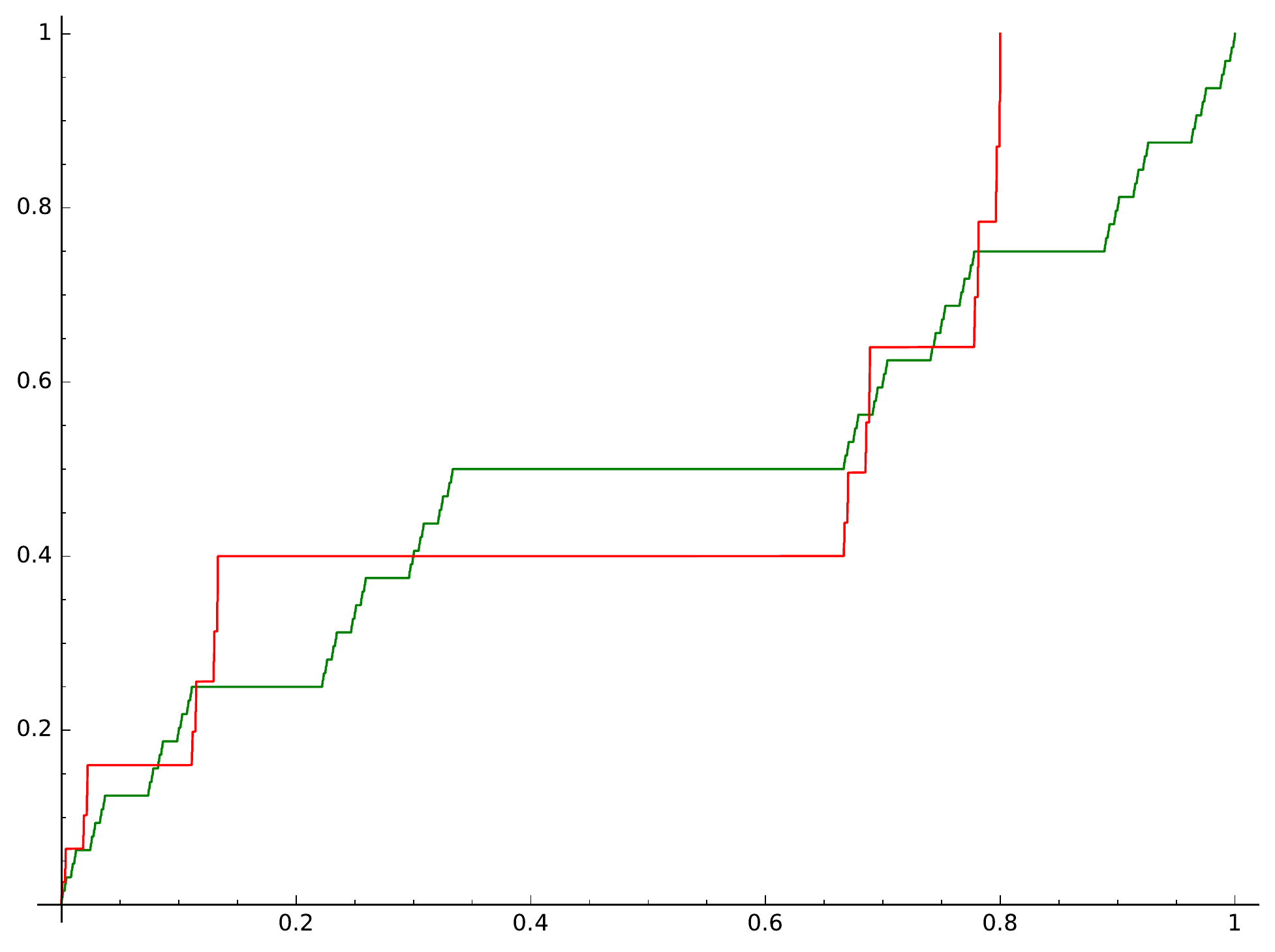}
\caption{E.g. 7. Cantor staircase for $\mu^{(f,p)}[0,x]$ in green and for $\mu^{(g,q)}[0,x]$ in red.}
\end{figure}

\begin{figure}[!]
\centering
\includegraphics[scale=0.33]{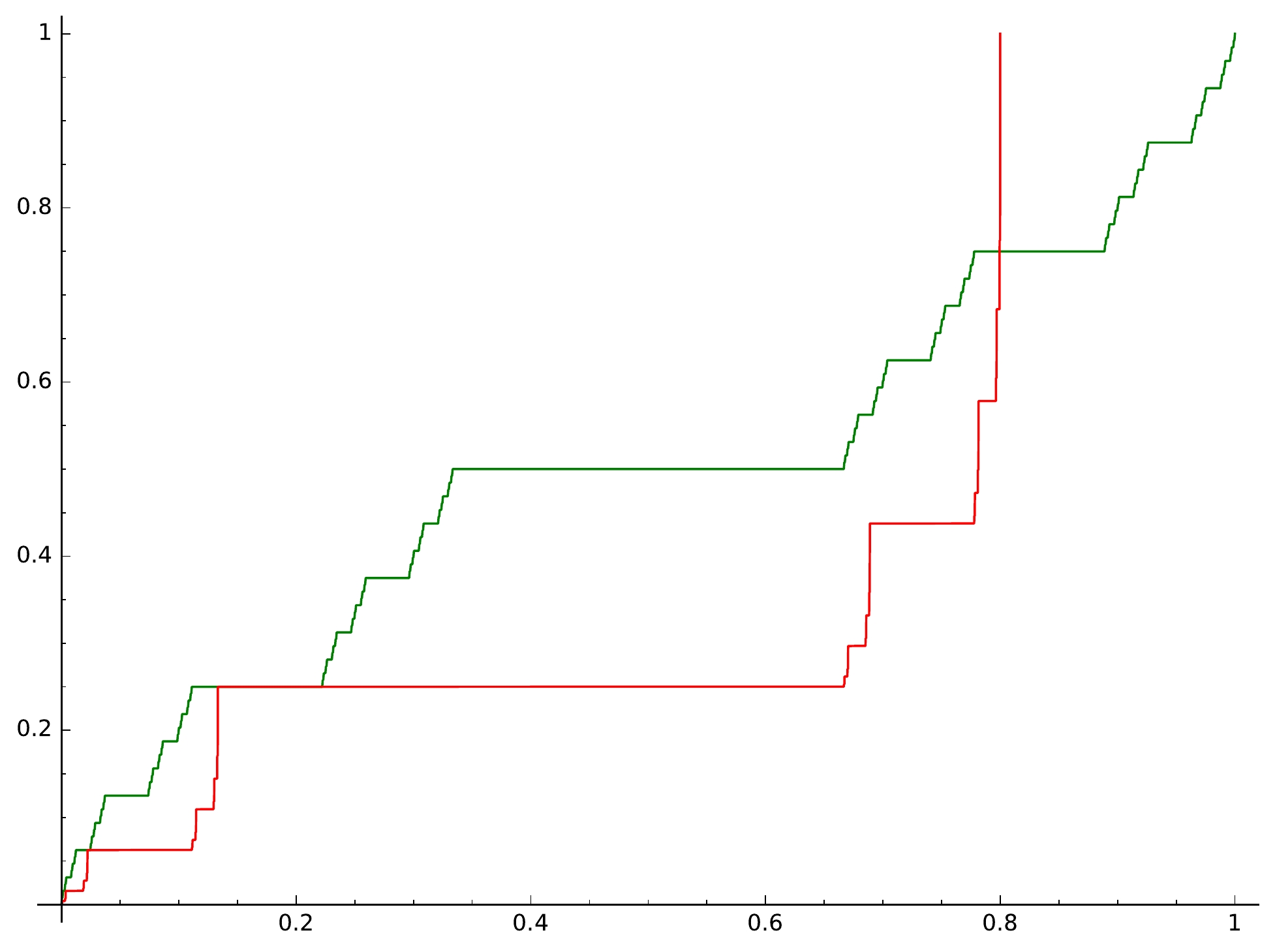}
\caption{E.g. 8. Cantor staircase for $\mu^{(f,p)}[0,x]$ in green and for $\mu^{(g,q)}[0,x]$ in red.}
\end{figure}


\begin{figure}[!]
\centering
\includegraphics[scale=0.33]{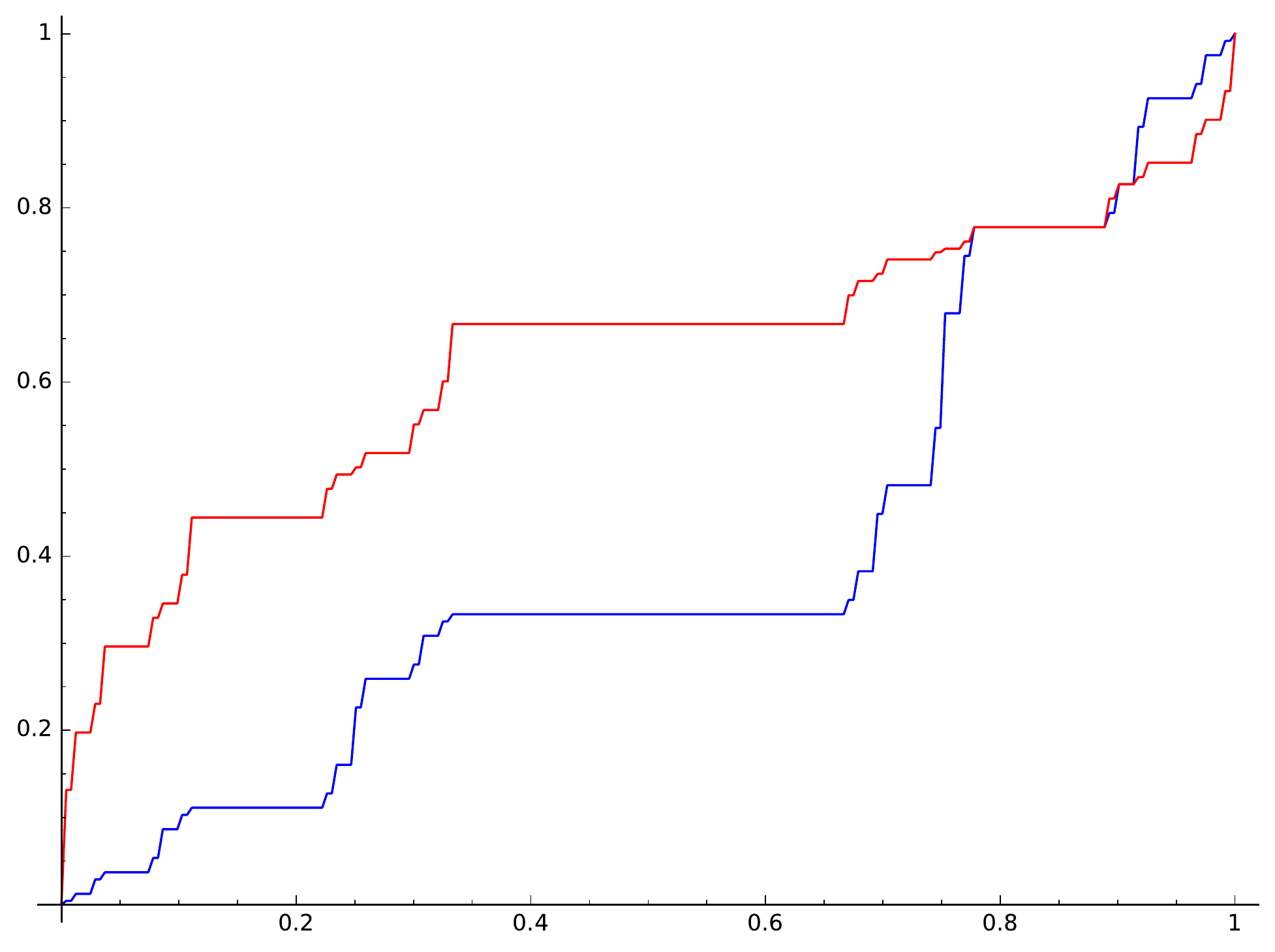}
\caption{E.g. 9. Cantor staircase for $\mu^{(f^{3},(1/3,2/3))}[0,x]$ in blue and for $\mu^{(f^{3},(2/3,1/3))}[0,x]$ in red.}
\end{figure}

\begin{figure}[!]
\centering
\includegraphics[scale=0.33]{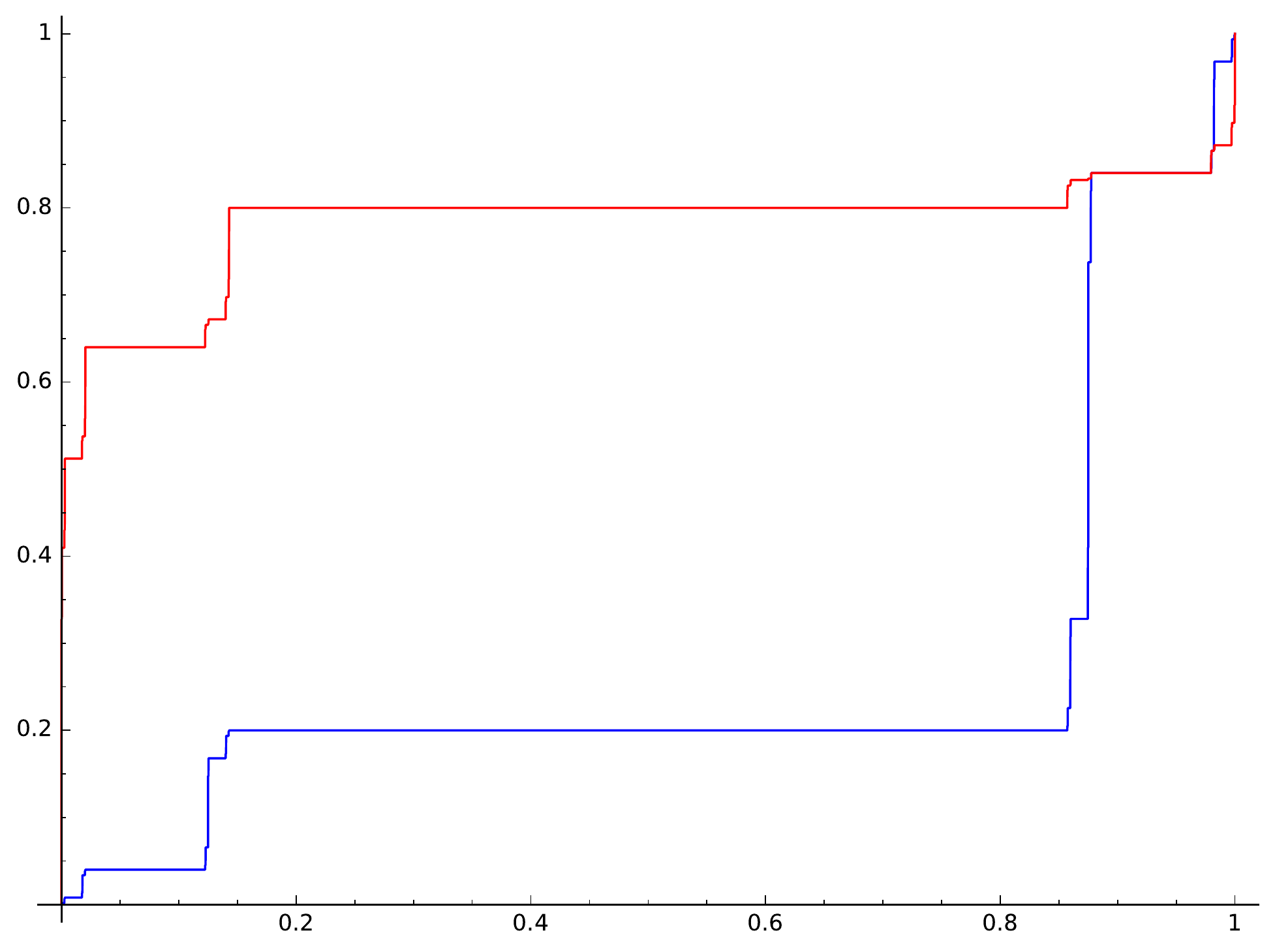}
\caption{E.g. 10. Cantor staircase for $\mu^{(f^{7},(1/5,4/5))}[0,x]$ in blue and for $\mu^{(f^{7},(4/5,1/5))}[0,x]$ in red.}
\end{figure}

\begin{figure}[!]
\centering
\includegraphics[scale=0.33]{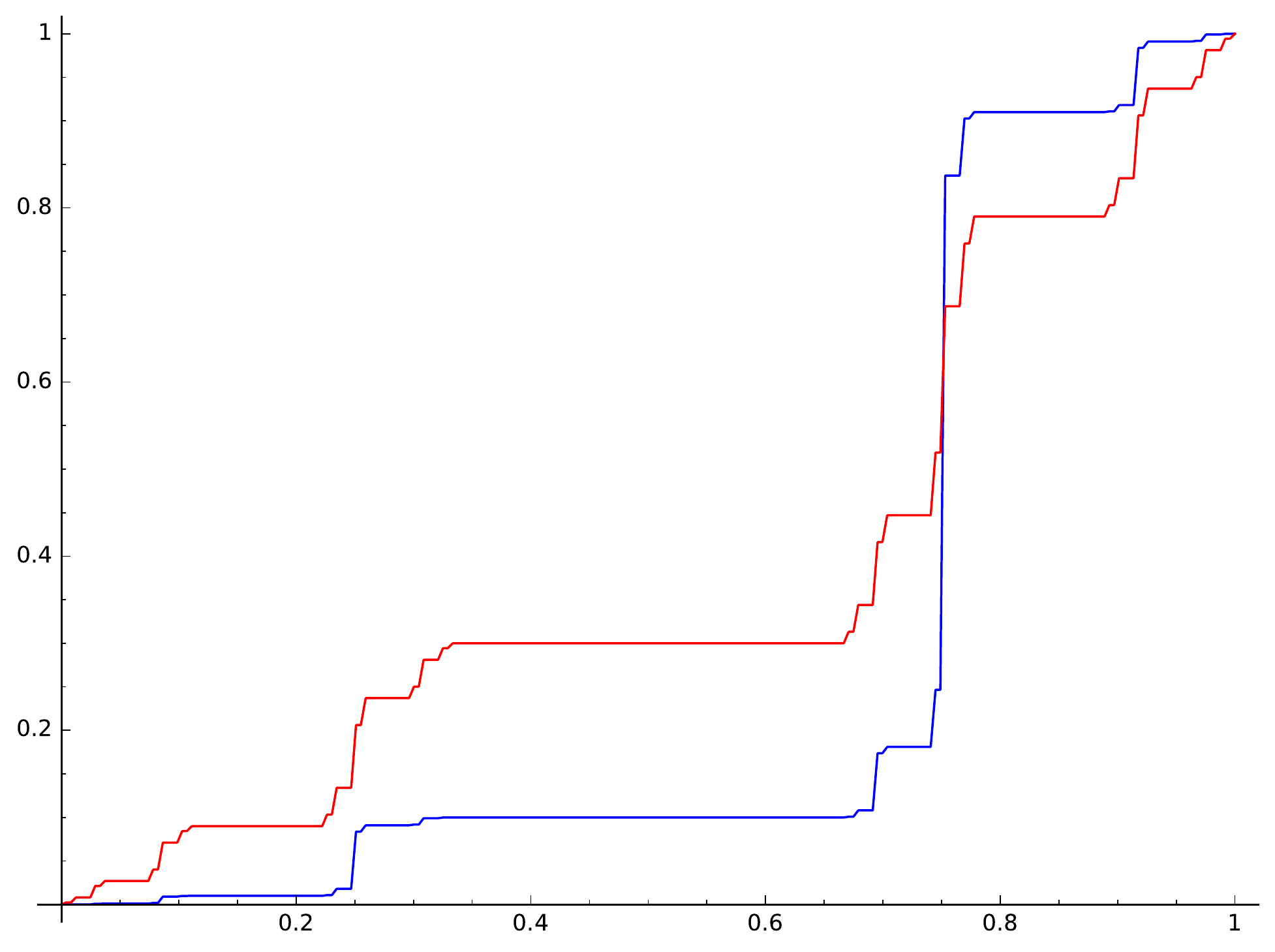}
\caption{E.g. 11. Cantor staircase for $\mu^{(f^{3},(0.1,0.9))}[0,x]$ in blue and for $\mu^{(f^{3},(0.3,0.7))}[0,x]$ in red.}
\end{figure}

\begin{figure}[!]
\centering
\includegraphics[scale=0.33]{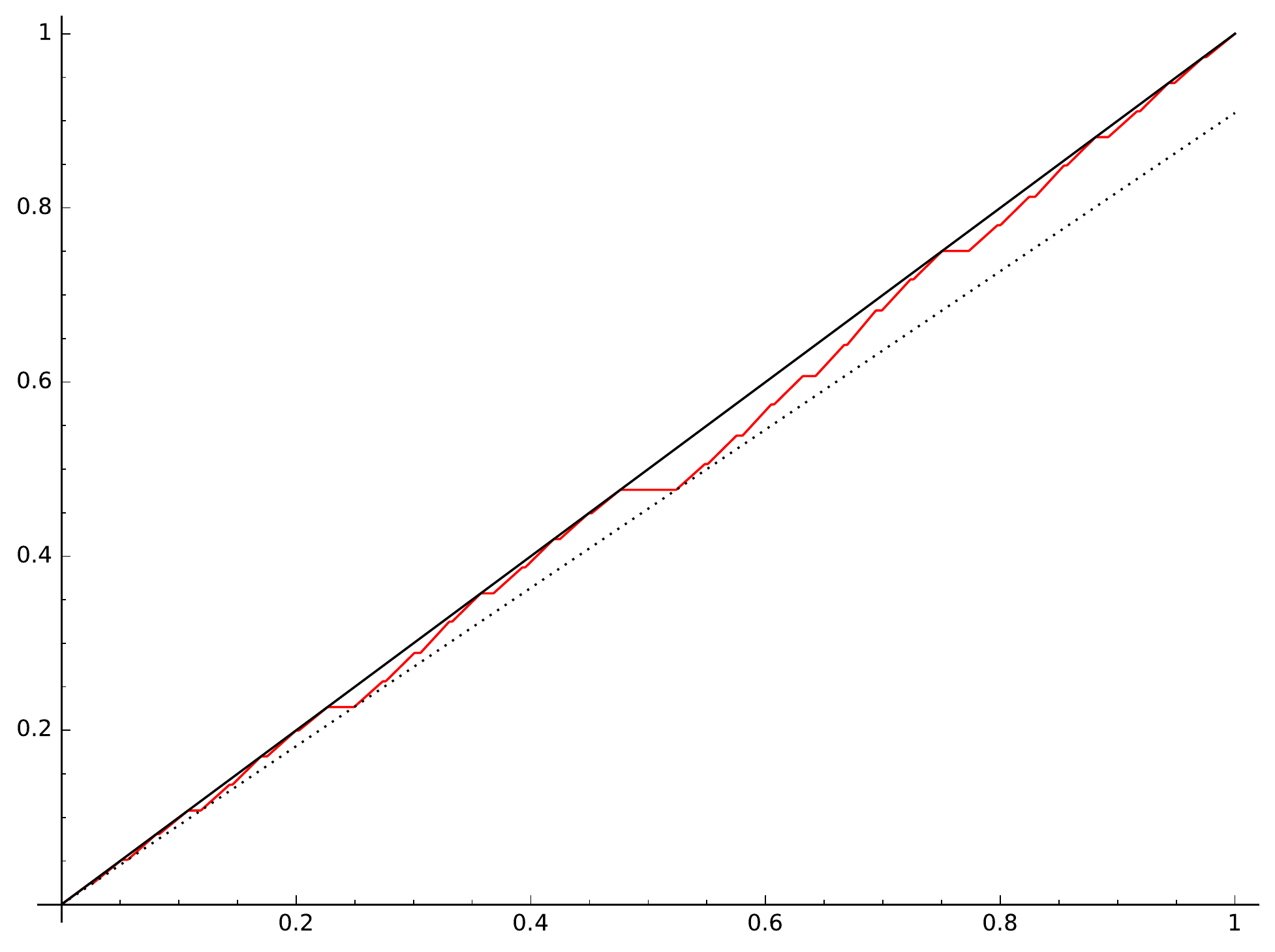}
\caption{E.g. 12. Cantor staircase for $\mu^{(f^{2.1},(1/2.1,1.1/2.1))}[0,x]$ in red and black lower and upper bounds.}
\end{figure}

\begin{figure}[!]
\centering
\includegraphics[scale=0.33]{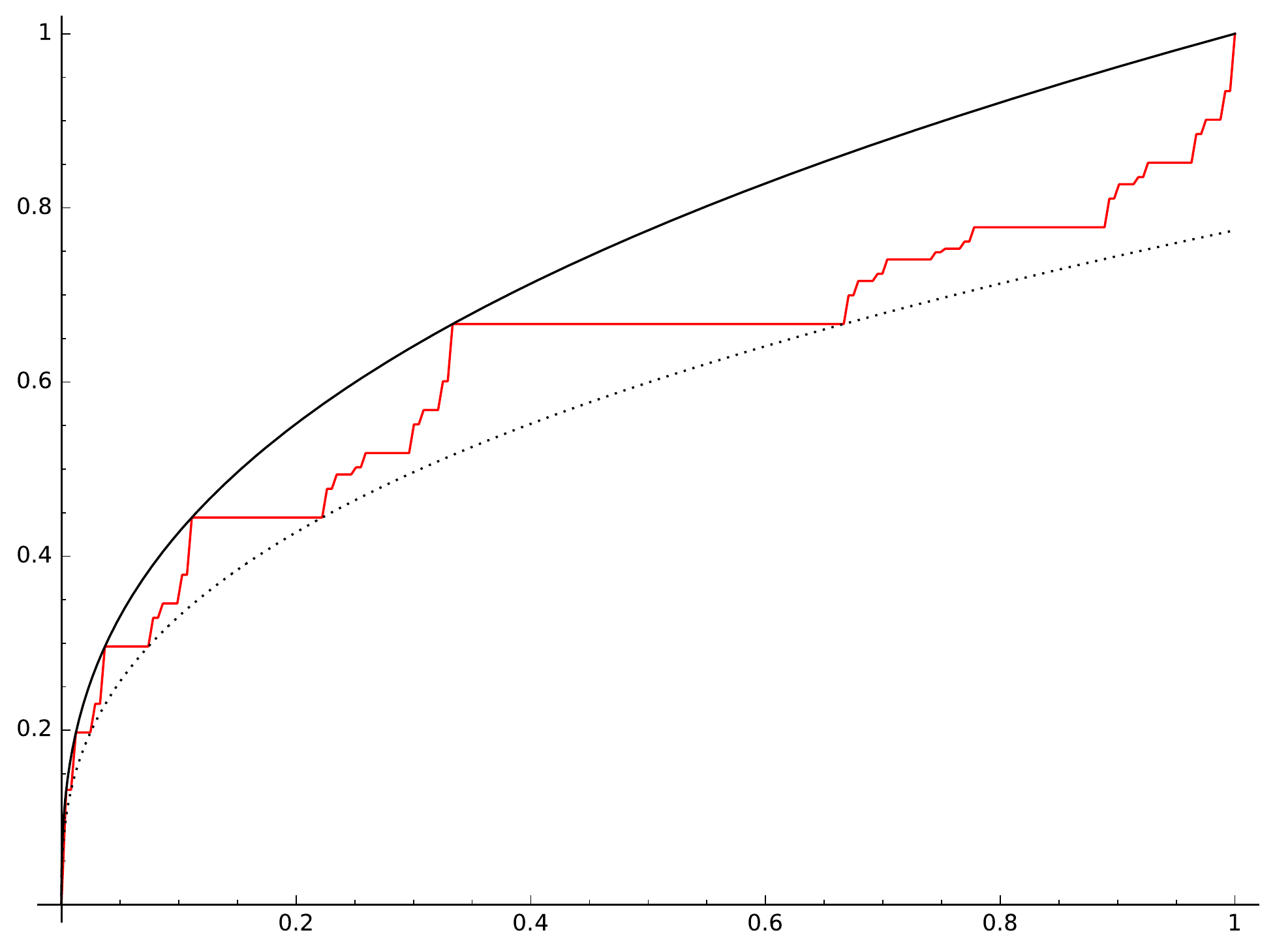}
\caption{E.g. 13.  Cantor staircase for $\mu^{(f^{3},(2/3,1/3))}[0,x]$ in red and black lower and upper bounds}
\end{figure}

\begin{figure}[!]
\centering
\includegraphics[scale=0.33]{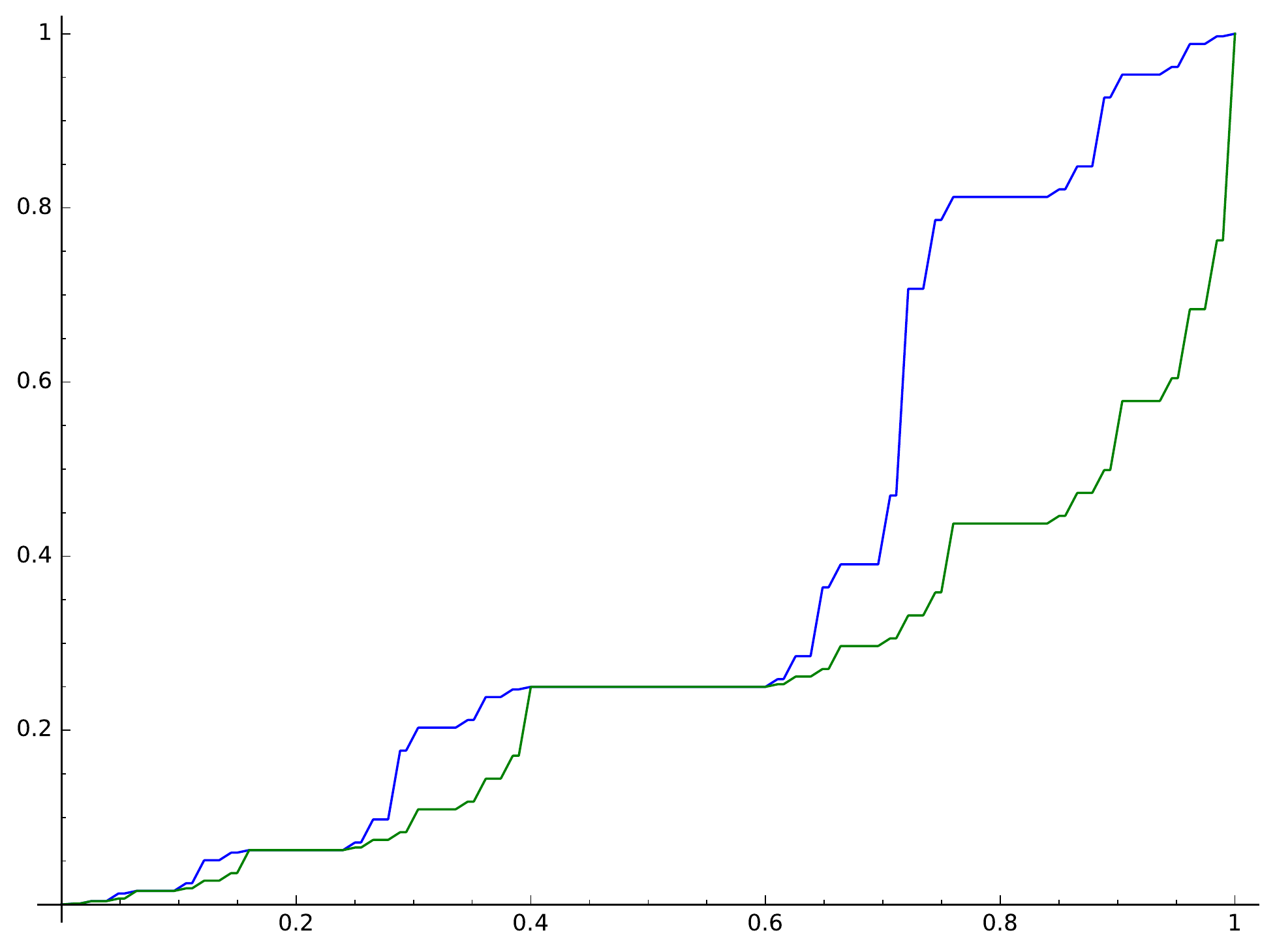}
\caption{E.g. 14. Cantor staircase for $\mu^{(f^{2.5},(1/4,3/4))}[0,x]$ in blue and for $\mu^{(g^{2.5},(1/4,3/4))}[0,x]$ in green.}
\end{figure}

\begin{figure}[!]
\centering
\includegraphics[scale=0.33]{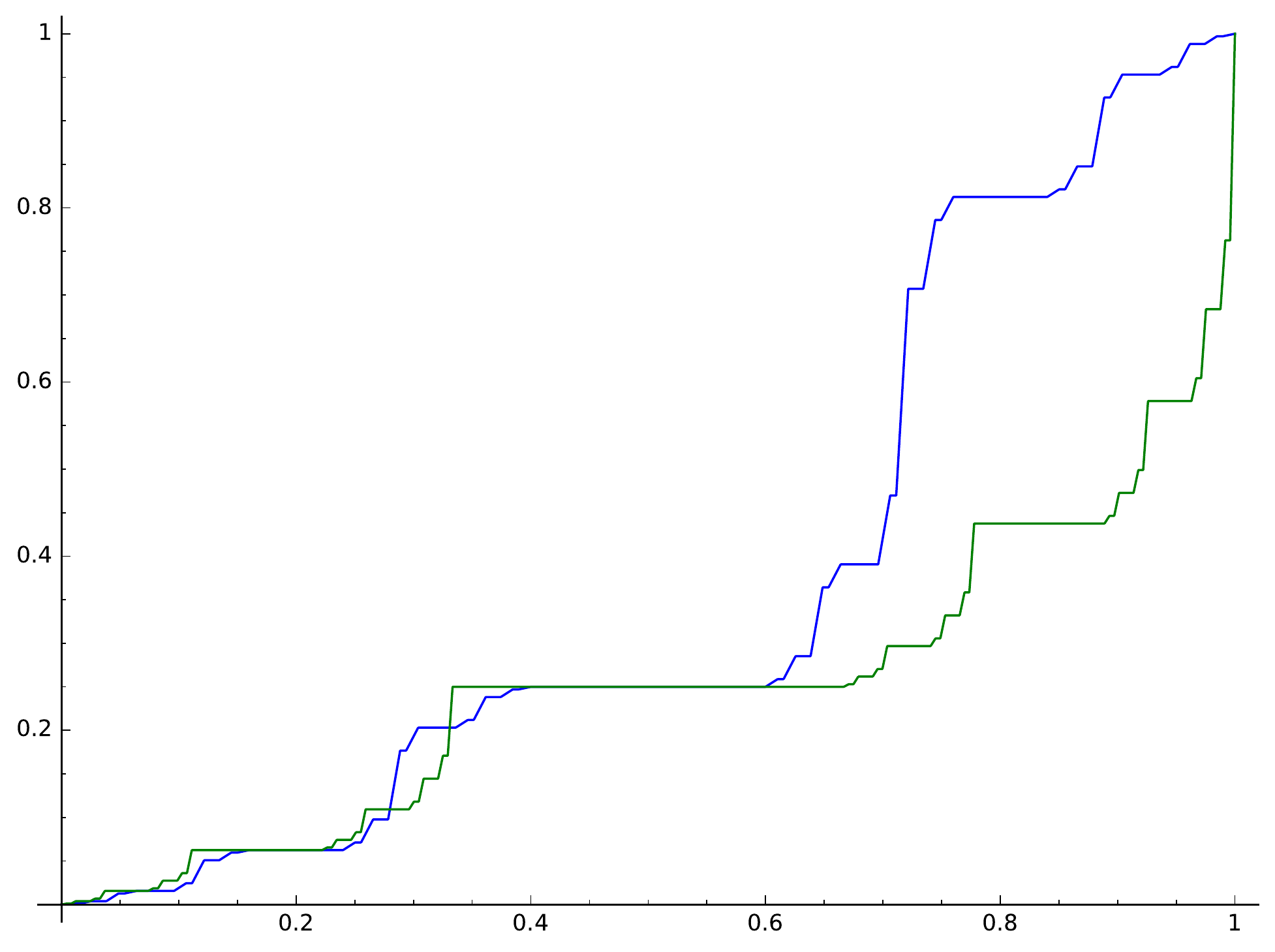}
\caption{E.g. 15. Cantor staircase for $\mu^{(f^{2.5},(1/4,3/4))}[0,x]$ in blue and for $\mu^{(g^{3},(1/4,3/4))}[0,x]$ in green.}
\end{figure}

\begin{figure}[!]
\centering
\includegraphics[scale=0.33]{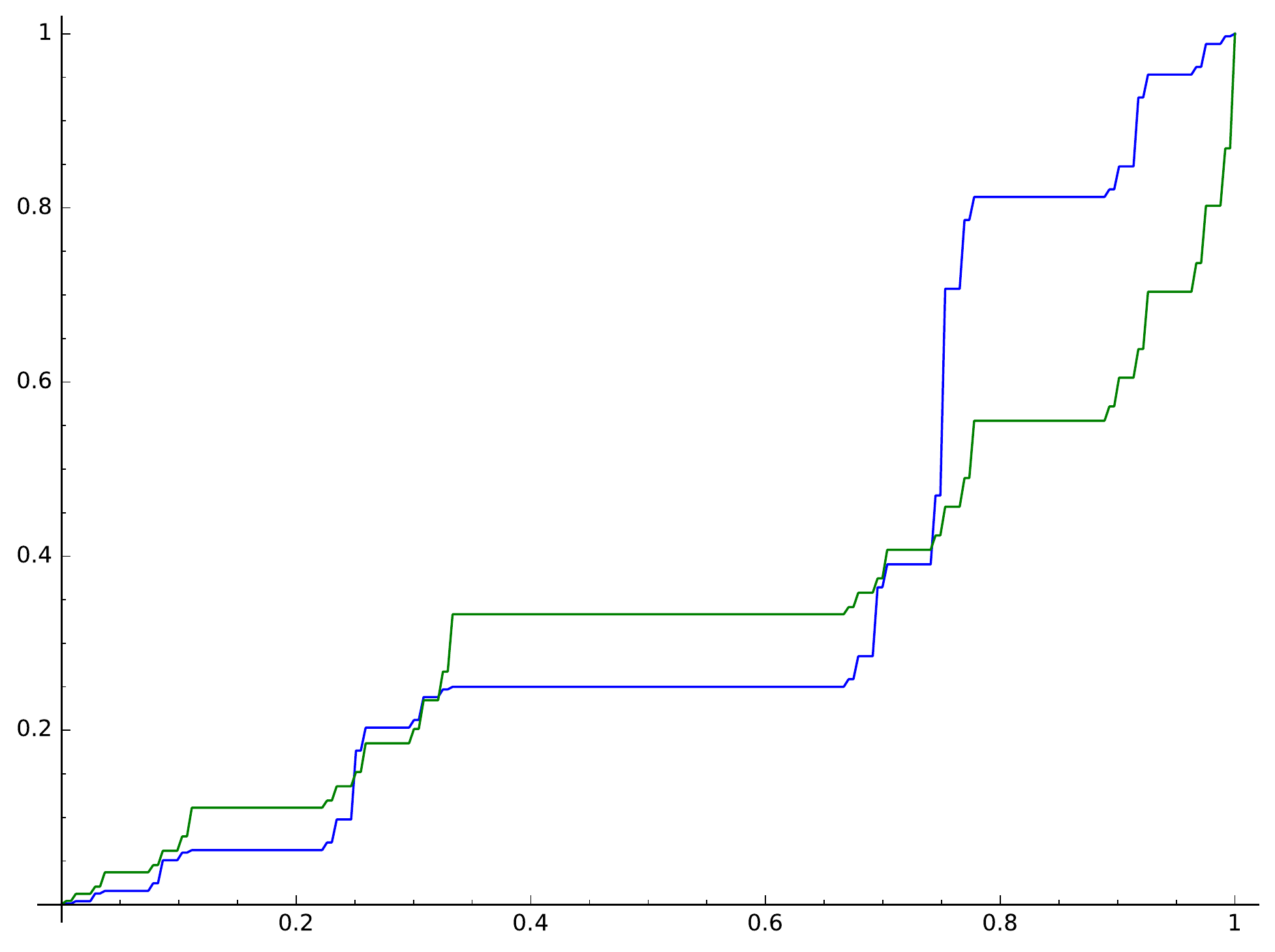}
\caption{E.g. 16. Cantor staircase for $\mu^{(f^{3},(1/4,3/4))}[0,x]$ in blue and for $\mu^{(g^{3},(1/3,2/3))}[0,x]$ in green.}
\end{figure}

\end{document}